\documentclass[11pt]{article}
\usepackage[T1]{fontenc}
\usepackage{lmodern}
\usepackage{amsmath,amssymb,amsthm,mathtools}
\usepackage{xcolor}
\usepackage{aliascnt}
\usepackage{booktabs,array}
\usepackage[margin=1.08in]{geometry}
\usepackage[hidelinks]{hyperref}
\usepackage[nameinlink,noabbrev]{cleveref}
\usepackage{microtype}
\hypersetup{
 pdftitle={A central limit theorem for the random assignment problem},
 pdfauthor={Gilles Mordant}
}

\numberwithin{equation}{section}
\newtheorem{theorem}{Theorem}[section]
\newaliascnt{proposition}{theorem}
\newtheorem{proposition}[proposition]{Proposition}
\aliascntresetthe{proposition}
\newaliascnt{lemma}{theorem}
\newtheorem{lemma}[lemma]{Lemma}
\aliascntresetthe{lemma}
\newaliascnt{corollary}{theorem}

\aliascntresetthe{corollary}
\theoremstyle{remark}
\newaliascnt{remark}{theorem}
\newtheorem{remark}[remark]{Remark}
\aliascntresetthe{remark}
\crefname{proposition}{proposition}{propositions}
\Crefname{proposition}{Proposition}{Propositions}
\crefname{lemma}{lemma}{lemmas}
\Crefname{lemma}{Lemma}{Lemmas}
\crefname{corollary}{corollary}{corollaries}
\Crefname{corollary}{Corollary}{Corollaries}
\crefname{remark}{remark}{remarks}
\Crefname{remark}{Remark}{Remarks}
\DeclareMathOperator{\Var}{Var}

\DeclareMathOperator{\osc}{osc}
\DeclareMathOperator{\diag}{diag}
\newcommand{\E}{\mathbb E}
\newcommand{\Pp}{\mathbb P}
\newcommand{\R}{\mathbb R}
\newcommand{\1}{\mathbf 1}

\newcommand{\dd}{\,\mathrm d}

\title{A central limit theorem for the random assignment problem}
\author{Gilles Mordant\thanks{Yale University}}
\date{\today}

\begin{document}
\maketitle

\begin{abstract}
Let \(C_n\) be the minimum cost of a perfect matching in an
\(n\times n\) matrix of independent uniform random variables.  We
 prove that
\[
 \sqrt n\{C_n-\zeta(2)\}
 \ \Longrightarrow\ 
 \mathcal N\bigl(0,4\zeta(2)-4\zeta(3)\bigr).
\]
The proof begins with an exact change of variables based on a
uniformly rooted shortest-path selection of an optimal dual
potential.  After the unused reduced costs are integrated out, a
reference law separates the rows conditionally on the potential
field, while the ordered potential gaps become independent
exponentials.  The only residual dependence is a directed-tree
factor.  Ordering the potentials turns its zero--one support into a
Ferrers matrix, whose matrix-tree determinant is triangular.  A
singular inverse-degree estimate and exact normalization then yield
total-variation convergence to the reference law.  Finally, a
conditional triangular-array central limit theorem accounts for row
noise, and a second triangular array accounts for the linear response
of the potential field.
The strategy used here is likely to be applicable to other problems.
\end{abstract}

\section{Introduction}
The random assignment problem is classical.  Let
\(Y=(Y_{ij})_{1\le i,j\le n}\) have independent
\({\rm Unif}[0,1]\) entries, where \(Y_{ij}\) is the cost of assigning
person \(i\) to task \(j\).  The minimum cost of a one-to-one assignment is
\[
 C_n=\min_{\sigma\in S_n}\sum_{i=1}^nY_{i,\sigma(i)},
\]
where $S_n$ is the set of permutations of $\{1, \ldots, n\}$. 
This problem has a long history (reviewed below), but a central limit
theorem has thus far remained out of reach.  We prove the following result.

\begin{theorem}\label{thm:main}
As \(n\to\infty\),
\[
 \sqrt n\{C_n-\zeta(2)\}
 \Longrightarrow
 \mathcal N\bigl(0,4\zeta(2)-4\zeta(3)\bigr).
\]
\end{theorem}

\subsection{AI disclosure}

This proof is not a one-prompt exploit: I have been working for quite some time on optimal transport and matching problems.
I somehow forced the AI to help me explore a geometric intuition that I had come up with a few months ago, even before the models reached their current level.
Funnily, during the interaction, I had to force the AI not to drift to attempts involving the Stein method and force it to stick to my ideas.

AI was then used to complete the proofs, catch mistakes and verify the paper (both via numerical simulations and general ``thinking''), as well as to improve the exposition.
The models ChatGPT 5.6 and Opus 5 (as well as previous versions) were used.

\subsection{History of the problem}
The probabilistic assignment problem has been studied for more than
six decades.  Kurtzberg's 1962 analysis gave logarithmic upper bounds
for simple assignment heuristics \cite{Kurtzberg1962}.  Walkup then
proved the first bound on \(\E C_n\) that was uniform in \(n\), namely
\(\E C_n<3\) \cite{Walkup1979}.  Linear programming soon became central:
Karp reduced the uniform upper bound to \(\E C_n<2\) via linear-programming
structure \cite{Karp1987}, while Dyer, Frieze
and McDiarmid placed this argument in a general theory of linear
programs with random costs \cite{DyerFriezeMcDiarmid1986}.  Steele's
monograph, especially its fourth chapter, gives a systematic account
of this early probabilistic-optimization line \cite{Steele1997}.
The complementary lower-bound line was already primal--dual in
spirit.  Lazarus obtained \(1+e^{-1}\); Goemans and Kodialam constructed
a sharper dual heuristic giving a constant greater than \(1.441\); and
Birgitta Olin's thesis constructed a feasible random dual yielding the
then-best bound \(1.51\)
\cite{Lazarus1993,GoemansKodialam1993,Olin1992}.  Thus both the
importance of the dual variables and a narrow interval around the
eventual answer were visible well before the exact solution.

A second line came from disordered systems.  M\'ezard and Parisi used
the replica method, first to predict the limit \(\zeta(2)=\pi^2/6\) and
then to analyze the replica-symmetric cavity solution and its
finite-size corrections \cite{MezardParisi1985,MezardParisi1987}.
Parisi later conjectured the striking finite-\(n\) identity
\[
 \E C_n^{\rm exp}=\sum_{k=1}^n\frac1{k^2}
\]
for mean-one exponential costs \cite{Parisi1998}.  Coppersmith and
Sorkin generalized the conjecture to minimum \(k\)-assignments in
rectangular matrices and, by a constructive algorithm, improved the
previous upper bound on the limiting mean to below \(1.94\)
\cite{CoppersmithSorkin1999}.  Exact small cases and Laplace transforms
were developed by Alm and Sorkin, while Buck, Chan and Robbins
formulated broader finite conjectures
\cite{AlmSorkin2002,BuckChanRobbins2002}.
Parisi and Rati\'eville subsequently corrected errors in the
M\'ezard--Parisi replica computation of the leading finite-size term,
and Caracciolo, D'Achille, Malatesta
and Sicuro studied how that term depends on the cost distribution
\cite{ParisiRatieville2002,CaraccioloDachilleMalatestaSicuro2017}.

The rigorous first-order theory developed in parallel.  Aldous first
proved convergence of the mean and of the optimum for exponential
costs, and observed that the uniform model is asymptotically equivalent
because only the density at zero matters \cite{Aldous1992}.  His objective method then
identified the constant as \(\zeta(2)\) through the Poisson-weighted
infinite tree (PWIT) \cite{Aldous2001}.  The full finite exponential
formula, including the Coppersmith--Sorkin generalization, was proved
by Linusson and W\"astlund and independently by the Stanford group of
Nair, Prabhakar and Sharma
\cite{LinussonWastlund2004,NairPrabhakarSharma2005}.  Recursive
distributional equations and endogeny clarified why the logistic PWIT
message is canonical
\cite{AldousBandyopadhyay2005,Bandyopadhyay2011}.  On the algorithmic
side, Bayati, Shah and Sharma related the assignment LP to convergent
max-product belief propagation, and Salez and Shah proved asymptotic
optimality in the random model with quadratic running time
\cite{BayatiShahSharma2008,SalezShah2009}.  

Fluctuations have proved substantially harder than the mean.  Talagrand
applied product-space concentration directly to random assignment
\cite{Talagrand1995}; for uniform costs, Lee and Su later obtained the
explicit bounds
\[
 c_1n^{-5/2}(\log n)^{-3/2}
 \le \Var(C_n)\le
 c_2n^{-1}(\log n)^2
\]
\cite{LeeSu2002}.  For mean-one exponential costs, W\"astlund found
exact formulas for all moments and proved
\[
 \Var(C_n^{\rm exp})
 =\frac{4\zeta(2)-4\zeta(3)}n+O(n^{-2})
\]
\cite{Wastlund2005,Wastlund2010}.  Chatterjee subsequently proved an
order-\(n^{-1/2}\) lower bound on fluctuations under his
\(\mathcal P^+(1)\) hypothesis.  Its smooth super-polynomial-tail
conditions exclude the compactly supported uniform law; he also noted
that an upper bound of order \(n^{-1/2}\) was not known for
nonexponential costs \cite{Chatterjee2019}.  Cao proved
CLTs for sparse optimization problems including diluted minimum
matching and pointed to the apparent loss of uniform contraction in the
mean-field limit as the obstruction for minimum matching \cite{Cao2021}.
In the nonbipartite random-link matching model,
Malatesta, Parisi and Sicuro used replicas and the cavity method to
derive the corresponding variance \((\zeta(2)-\zeta(3))/n\) and a
large-deviation function \cite{MalatestaParisiSicuro2019}; the factor
of four agrees with W\"astlund after passing between matching and
assignment normalizations, but this is not a proof for the bounded
bipartite model.  Most recently, W\"astlund showed that the exponential
model has an explicitly recursive rational moment-generating function;
his proposed zero-free-disk conjecture would imply a Gaussian limit,
but remains open \cite{Wastlund2026}.  Thus none of the first-order,
concentration, exact-moment or replica results supplies the central
limit theorem for bounded uniform costs.  The local PWIT describes a
typical neighborhood and its cavity message, whereas the theorem below
also requires the joint order-\(n^{-1/2}\) motion of all dual messages.
This global response is the second Gaussian component in our proof.

\subsection{Proof outline and strategy}

The assignment linear program and its dual are
\begin{equation}\label{eq:PrimDual}
 \min\Bigl\{\sum_{i,j}Y_{ij}x_{ij}:x\mathbf1=\mathbf1,
 x^{\mathsf T}\mathbf1=\mathbf1,\ x\ge0\Bigr\},
 \quad\text{and}\quad
 \max\Bigl\{\sum_i\alpha_i+\sum_j\beta_j:
 \alpha_i+\beta_j\le Y_{ij}\Bigr\},
\end{equation}
respectively.  We first relabel the almost surely unique optimal
matching as the diagonal.  An optimal dual pair can then be written
\[
 \alpha_i=d_i-q_i,\qquad \beta_i=q_i,\qquad d_i=Y_{ii},
\]
and dual feasibility becomes
\begin{equation}\label{eq:intro-dual}
 Y_{ij}\ge d_i+q_j-q_i.
\end{equation}
For a uniformly chosen root, shortest paths in the replacement-cost
graph select one feasible \(q\), without declaring a deterministic
vertex special.  With
\[
 D_i=nd_i,\qquad U_i=n(q_i-\bar q),\qquad A_i=D_i-U_i,
\]
the exact primal--dual identities are
\begin{equation}\label{eq:minplus}
 A_i=\min_j\{nY_{ij}-U_j\},\qquad
 U_j=\min_i\{nY_{ij}-A_i\},\qquad
 D_i=A_i+U_i,\qquad
 C_n=\frac1n\sum_iD_i=\frac1n\sum_iA_i.
\end{equation}
Thus \(U_j\) is the relative shadow price of target \(j\), \(A_i\)
is the response of source \(i\) to the entire target-price field, and
\(D_i\) is the cost of the selected edge.

The proof's first idea is to construct a joint representation of the
costs retained by the optimizer together with a dual certificate of
optimality.  Selecting a root uniformly at random, a shortest-path
rule chooses one optimal dual potential and gives an exact change of
variables on the optimizer cell.  This change of variables is the key
to Proposition~\ref{prop:canonical}.  The latter states that, on
\(0\le D_i\le n\), the joint distribution of \((D,U)\) is
proportional to
\[
 \prod_{i\ne j}
 \left(1-\frac{(D_i+U_j-U_i)_+}{n}\right)_+
 \mathfrak T(w),
\]
where \(\mathfrak T(w)\) is a sum of directed-tree polynomials and
\(w\) depends on the matrix
\((D_i+U_j-U_i)_{1\le i,j\le n}\).  Its precise definition is not
important at this point.  Borrowing terminology from
the cavity method, \(U\) plays the role of the cavity field, while the
\(A_i\)'s are the responses of the sources to this field.  This part
of the proof is carried out in Section~\ref{sec:chart}.

\medskip

The exact joint distribution of \((D,U)\) is still too difficult to
analyze directly.  We therefore approximate it, and then approximate
it once more, at asymptotically negligible cost in total variation.

For each fixed \(U\), let \(\nu_{i,U}\) be the normalized distribution of the value at
which the first constraint in row \(i\) becomes tight.  The adaptive
law is
\[
 \mathbb P_n^{\mathrm{ad}}(\dd U,\dd A)
 =
 (Z_n^{\mathrm{ad}})^{-1}
 \prod_iZ_i(U)\,\dd U
 \bigotimes_{i=1}^n\nu_{i,U}(\dd A_i).
\]
Under this law, the \(A_i\)'s are independent conditional on \(U\),
although their conditional distributions still depend on the common
field.

Replacing the marginal law of \(U\) by its leading interaction gives
the reference law $ \mathbb P_n^{\mathrm{ref}}$.   Its potential density is proportional to
\[
 \exp\left\{-\frac1n\sum_{i<j}|U_i-U_j|\right\}.
\]
This is the law of a one-dimensional attractive gas.  Once the
potentials are ordered, the consecutive gaps are independent
exponentials, from which a logistic empirical law emerges.
Section~\ref{sec:reference} defines these two approximating laws and
proves
\(
 \|\mathbb P_n^{\mathrm{ad}}
       -\mathbb P_n^{\mathrm{ref}}\|_{\mathrm{TV}}
 \longrightarrow0.
\)

\medskip

Section~\ref{sec:rows-short} studies one row conditionally on \(U\).
The product of the residual constraint lengths gives the survival
function of \(D_i\), while its logarithmic derivative gives the rate
at which one of the row constraints becomes tight.  This
first-failure representation yields the moment and tail estimates
needed for a conditional characteristic-function argument.  The same
section identifies the limiting empirical vertex law and controls
rows with unusually few active constraints.

\medskip

It remains to justify replacing the exact canonical law by the
adaptive conditional product.  This is the purpose of
Section~\ref{sec:ferrers-short}.  After normalizing the tree
polynomial by its row degrees, the ordered potentials give the
corresponding matrix a nested Ferrers support.  The matrix-tree
determinant can then be triangularized, and the normalized tree
likelihood is shown to converge in \(L^1\) to a constant.  This proves
Proposition~\ref{prop:second-tv-short} that states
\(
 \|\mathbb P_n^{\mathrm{can}}
       -\mathbb P_n^{\mathrm{ad}}\|_{\mathrm{TV}}
 \longrightarrow0.
\)
Together with the first comparison,
\(
 \|\mathbb P_n^{\mathrm{can}}
       -\mathbb P_n^{\mathrm{ref}}\|_{\mathrm{TV}}
 \longrightarrow0.
\)

\medskip

Section~\ref{sec:fluct-short} identifies the two fluctuations that
remain under the reference law.  Write
\[
 m_{i,n}=\E(D_i\mid U).
\]
Then
\begin{equation}\label{eq:intro-decomp}
 \sqrt n\left\{\frac1n\sum_iD_i-\zeta(2)\right\}
 =
 \frac1{\sqrt n}\sum_i(D_i-m_{i,n})
 +
 \frac1{\sqrt n}
 \left\{\sum_im_{i,n}-n\zeta(2)\right\}.
\end{equation}
The first term in \eqref{eq:intro-decomp} is the fluctuation of the
rows with the cavity field \(U\) held fixed.  The second term requires
a second cavity step: one must let the field itself fluctuate and
measure the resulting change in the aggregate row response.  If
\[
 \varpi_n=\frac1n\sum_i\delta_{U_i},
\]
the sum of the conditional means is approximated, on the
\(\sqrt n\)-scale, by a functional
\[
 \Phi_n(U)=n\mathcal M(\varpi_n).
\]
The first variation of \(\mathcal M\) at the logistic law gives its
influence function.  At finite \(n\), this is implemented by
linearizing \(\Phi_n\) around the deterministic harmonic grid
\(\bar u_n\):
\[
 \Phi_n(U)-\Phi_n(\bar u_n)
 =
 D\Phi_n(\bar u_n)[U-\bar u_n]+o_{\mathbb P}(\sqrt n).
\]
Since the ordered fluctuations of \(U\) are generated by independent
exponential gaps, the linear term becomes a triangular array.  This
linear response of the row average to the moving cavity field is the
second-order cavity contribution.
The first term is conditionally independent row noise.  The second is
the response of the conditional mean to the moving dual field.  The
first is handled by the conditional characteristic-function
expansion, while the second is reduced to a triangular array of the
independent exponential gaps.  Conditional characteristic functions
combine the two terms without requiring their unconditional
independence.

Finally, Section~\ref{sec:ProofComp} proves the central limit theorem
under \(\mathbb P_n^{\mathrm{ref}}\) and transfers it to
\(\mathbb P_n^{\mathrm{can}}\) using the total-variation comparisons.
Under the exact canonical law,
\(C_n=n^{-1}\sum_iD_i\), which completes the proof for the optimal
assignment cost.

\section{The rooted dual chart and its exact density}
\label{sec:chart}

\subsection{Dual feasibility as a shortest-path problem}

Continuity of the entries makes the optimal permutation unique almost
surely.  After relabeling the columns of the cost matrix, the optimal permutation is the identity and
complementary slackness gives \eqref{eq:intro-dual}.  Put
\(\Delta_{ij}=Y_{ij}-d_i\), \(i\ne j\).  The diagonal is optimal if
and only if every directed cycle has nonnegative \(\Delta\)-length:
the cost difference of a competing permutation is the sum of the
lengths of its nontrivial cycles.

Fix a root \(o\) and let \({\rm dist}_o(i)\) be the shortest
\(\Delta\)-length of a directed path from \(i\) to \(o\).  There are
no negative cycles, so this is finite, and
\[
 q_i^{(o)}=-{\rm dist}_o(i),\qquad q_o^{(o)}=0
\]
satisfies \eqref{eq:intro-dual}.  It is the componentwise least
feasible potential in the root gauge: summing the inequalities for
any other feasible potential along a path to \(o\) proves the claim.
Away from a null set of path ties, the first edges of the shortest
paths form an in-arborescence directed toward \(o\).

Here an in-arborescence toward \(o\) is a directed spanning tree in
which every nonroot vertex has one outgoing edge, the root has none,
and following arrows always reaches \(o\).  Thus the selected edge at
a vertex is the first step of its shortest route \emph{to} the root.
This convention fixes the orientation used by the directed
matrix-tree theorem later.

\subsection{A three-by-three handle on the construction}

The following example contains the finite construction without its
measure-theoretic bookkeeping.  Consider the already aligned matrix
\[
 Y=
 \begin{pmatrix}
  .10&.20&.45\\
  .40&.20&.30\\
  .45&.33&.15
 \end{pmatrix},
 \qquad d=(.10,.20,.15).
\]
The replacement increments are
\[
 \Delta_{ij}=Y_{ij}-d_i,\qquad
 \Delta=
 \begin{pmatrix}
 0&.10&.35\\
 .20&0&.10\\
 .30&.18&0
 \end{pmatrix}.
\]
The three directed two-cycles have lengths \(.30,.65,.28\), and the
two directed three-cycles have lengths \(.50,.73\).  They are all
positive, so the diagonal is the unique optimizer.  Choose root
\(o=3\).  The shortest route from \(2\) is \(2\to3\), of length
\(.10\), and that from \(1\) is \(1\to2\to3\), of length \(.20\).
Consequently
\[
 q=(-.20,-.10,0),\qquad
 \alpha=d-q=(.30,.30,.15),\qquad \beta=q.
\]
The matrix of dual lower bounds is
\[
 (\alpha_i+\beta_j)_{ij}=
 \begin{pmatrix}
  .10&.20&.30\\
  .10&.20&.30\\
  -.05&.05&.15
 \end{pmatrix}.
\]
It agrees with \(Y\) on the diagonal and on the two tree edges
\(1\to2\) and \(2\to3\), and is strict elsewhere.  The selected
potential is not claimed to be the only optimal potential.  For
example, \(\widetilde q=(-.15,-.05,0)\) is also feasible, but the
shortest-path choice satisfies \(q\le\widetilde q\) coordinatewise.

Centering and scaling give
\[
 U=3(q-\bar q\mathbf1)=(-.30,0,.30),\quad
 D=3d=(.30,.60,.45),\quad A=D-U=(.60,.60,.15).
\]
The net-offer matrix \(3Y_{ij}-U_j\) is
\[
 \begin{pmatrix}
  .60&.60&1.05\\
  1.50&.60&.60\\
  1.65&.99&.15
 \end{pmatrix}.
\]
Its row minima are exactly \(A\).  In the first two rows the diagonal
target ties with the outgoing tree target; these extra tight
inequalities are what select the rooted least potential.  Adding the
diagonal target prices back gives \(A+U=D\).

\subsection{Charts and their Jacobian}

For \(E=\{(i,j):i\ne j\}\), write
\[
 \theta_{ij}(q)=d_i+q_j-q_i,\qquad
 I_{ij}(q)=\1_{\{0<\theta_{ij}(q)<1\}},\qquad
 \lambda_{ij}(q)=(1-\theta_{ij}(q)_+)_+.
\]
Let \(\mathfrak T_o\) be the in-arborescences toward \(o\), and put
\[
\tau_o(z)=\sum_{T\in\mathfrak T_o}\prod_{e\in T}z_e,
\]
 which is a matrix-tree polynomial.
For a potential $q$, $\theta_{ij}(q)$ is the dual lower-bound threshold,
and $\lambda_{ij}(q)$ is the length of the fluctuation interval of entry
$(i,j)$ allowed by the dual inequality.

\begin{lemma}[Rooted charts]\label{lem:charts}
Fix \(d\in[0,1]^n\), let \(\mathcal C(d)\subset[0,1]^E\) be the
section on which the diagonal is optimal, and fix a root \(o\).  For
\(T\in\mathfrak T_o\), define
\[
 \Omega_{o,T}(d)=\left\{(q,y):q_o=0,\ 
 \begin{array}{ll}
 0<\theta_e(q)<1,&e\in T,\\
 \max\{0,\theta_e(q)\}<y_e<1,&e\notin T
 \end{array}\right\}
\]
and set
\[
 \Psi_{o,T,d}(q,y)_e=
 \begin{cases}
  \theta_e(q),&e\in T,\\
  y_e,&e\notin T.
 \end{cases}
\]
There is a null set \(\mathcal N_{d,o}\) such that the maps
\(\Psi_{o,T,d}\), as \(T\) varies, are bijections onto pairwise
disjoint subsets whose union is
\(\mathcal C(d)\setminus\mathcal N_{d,o}\).  Each chart has absolute
Jacobian one.  The set
\(\{(d,c):c\in\mathcal N_{d,o}\}\) is jointly null.  Consequently the
volume of the section at \(d\) is
\begin{equation}\label{eq:cross-section}
 \int_{\R^{n-1}}\sum_{T\in\mathfrak T_o}
 \prod_{e\in T}I_e(q)\prod_{e\notin T}\lambda_e(q)\,\dd q,
\end{equation}
where the dependence on $d$ has been suppressed at most places.
\end{lemma}

\begin{proof}
We give all exclusions explicitly.  Let \(\mathcal N_{d,o}\) be the
finite union of the affine sets on which
\begin{enumerate}
\item some off-diagonal coordinate has \(c_e\in\{0,1\}\);
\item a directed simple cycle has zero \(\Delta\)-length; or
\item two distinct directed simple paths from some vertex \(i\) to
the root \(o\) have the same \(\Delta\)-length.
\end{enumerate}
There are only finitely many simple cycles and paths.  Every equality
is nontrivial because the alternatives have distinct edge-incidence
vectors, so this is a null set.  The same finite list consists of
Borel affine equalities as \(d\) varies; Fubini therefore gives the
joint null-set assertion.

Take \(c\in\mathcal C(d)\setminus\mathcal N_{d,o}\).  Zero cycles
have been removed, so all directed cycles have strictly positive
length.  Every shortest walk can therefore be shortened to a simple
path, and the path from each \(i\) to \(o\) is unique.  Put
\(q_i=-\operatorname{dist}_o(i)\).  The shortest-path comparison gives
\[
 q_j-q_i\le c_{ij}-d_i.
\]
The paths are suffix-consistent: if their route from \(i\) first
visits \(j\), the remaining suffix must be the unique shortest path
from \(j\) to \(o\), or replacing it would improve the path from
\(i\).  Their first edges consequently form one in-arborescence
\(T\).  Each tree inequality is tight,
\(c_{ij}=d_i+q_j-q_i\).  Every nontree inequality is strict as well.
Indeed, if equality held on a nontree edge \(i\to j\), that edge
followed by the tree path from \(j\) to \(o\) would have total length
\(-q_i\).  If this walk is simple, it is a shortest path from \(i\)
different from the tree path, contrary to exclusion~(3).  If it
revisits \(i\), the edge \(i\to j\) together with the tree segment
from \(j\) back to \(i\) is a directed cycle of length
\((q_j-q_i)+(q_i-q_j)=0\), contrary to exclusion~(2).  The cube
boundary exclusion then gives
\[
 0<\theta_e(q)<1\quad(e\in T),\qquad
 \max\{0,\theta_e(q)\}<c_e<1\quad(e\notin T).
\]
Thus \(c\) lies in the image of exactly the chart indexed by its
shortest-path tree.  Moreover \(q\) is recovered from \(c\) as the
negative distance vector, \(T\) from the unique first edges, and
\(y=c_{E\setminus T}\).  This proves injectivity and disjointness,
not merely coverage.

Conversely, take \((q,y)\in\Omega_{o,T}(d)\) and put
\(c=\Psi_{o,T,d}(q,y)\).  Summing the dual inequalities along any
path from \(i\) to \(o\) gives length at least \(-q_i\).  The
\(T\)-path has exactly that length, while every other simple path
uses a strict nontree edge and is longer.  Hence
\(q_i=-\operatorname{dist}_o(i)\), the paths are unique, and their
first-edge tree is \(T\).  This proves surjectivity onto the claimed
piece.

Now observe that 
\[
\frac{\partial  (c_T,c_{E\setminus T})}{\partial (q_{-o},y_{E\setminus T})} = \begin{pmatrix}B_{T,o}&0\\0&I\end{pmatrix},
\]
where \(B_{T,o}\) is a reduced directed incidence matrix.  It has
determinant \(\pm1\): remove a nonroot leaf and expand along its unique
incident tree row, then iterate.  If \(d\) is included among the
coordinates, the full derivative remains block triangular with an
additional identity block.  Integrating each free nontree coordinate
over \((\max\{0,\theta_e\},1)\) gives one factor \(\lambda_e\) and
proves \eqref{eq:cross-section}.
\end{proof}

\subsection{Randomizing the root and centering the gauge}
\label{sec:RandRoot}
Pass to scaled, centered coordinates
\[
 D_i=nd_i,\quad U_i=n(q_i-\bar q),\quad A_i=D_i-U_i,
 \quad H_0=\{U:\textstyle\sum_iU_i=0\}.
\]
Set
\[
 \rho_n(t)=\left(1-\frac{t_+}{n}\right)_+,\qquad
 \omega_n(t)=
 \begin{cases}
  (1-t/n)^{-1},&0\le t<n,\\
  0,&\text{otherwise},
 \end{cases}
\]
\[
 a_{ij}=D_i+U_j-U_i=A_i+U_j,\qquad
 w_{ij}=\omega_n(a_{ij})\1_{\{i\ne j\}},\qquad
 \mathfrak T(w)=\sum_o\tau_o(w).
\]
These quantities have a direct boundary-crossing meaning.  For an
off-diagonal entry, dual feasibility requires the scaled cost
\(nY_{ij}\) to lie above \(a_{ij}\).  If \(0\le a_{ij}<n\), its
available interval inside \([0,n]\) has relative length
\(\rho_n(a_{ij})=(n-a_{ij})/n\).  Raising \(D_i\), and hence
\(a_{ij}\), by \(\dd d\) removes the fraction
\[
 \frac{\dd d}{n-a_{ij}}=\frac{w_{ij}}n\,\dd d
\]
of that interval.  Thus \(w_{ij}/n\) is the instantaneous rate at
which constraint \((i,j)\) becomes tight during a sweep of row \(i\).
If \(a_{ij}<0\), the lower boundary of the cost cube is still the
binding one, so a small increase does not yet threaten that constraint
and \(w_{ij}=0\).  Notice that \(w_{ij}\) is the inverse of the
\emph{remaining admissible interval}.
The Lebesgue measure on \(H_0\) means \(\dd U_1\cdots\dd U_{n-1}\)
with \(U_n=-\sum_{i<n}U_i\).

\begin{lemma}[Boundary convention]\label{lem:boundary-version}
Put \(\chi_e=\1_{\{0\le a_e<n\}}\).  A version of the rooted-chart
pushforward density is
\[
 g_{o,T}(D,U)=
 \prod_{e\in T}\chi_e\prod_{e\notin T}\rho_n(a_e).
\]
The only points at which this differs from the strict open-chart
description lie in the finite union of affine sets
\(\{a_e=0\}\cup\{a_e=n\}\), which is null in \(\R^n\times H_0\).
With the displayed non-strict convention,
\begin{equation}\label{eq:tree-factorization-global}
 \sum_o\sum_{T\in\mathfrak T_o}g_{o,T}(D,U)
 =\left\{\prod_{e\in E}\rho_n(a_e)\right\}\mathfrak T(w)
\end{equation}
pointwise, with no \(0/0\) interpretation.
\end{lemma}

\begin{proof}
If every residual length \(\rho_n(a_e)\) is positive, factor it from
each tree monomial:
\[
 \prod_{e\in T}\chi_e\prod_{e\notin T}\rho_n(a_e)
 =\left\{\prod_{e\in E}\rho_n(a_e)\right\}
   \prod_{e\in T}\frac{\chi_e}{\rho_n(a_e)}.
\]
The quotient is exactly \(w_e\), and summing gives
\eqref{eq:tree-factorization-global}.  If some residual length is
zero, then \(a_e\ge n\) and \(\chi_e=0\).  A tree containing \(e\)
has a zero \(\chi_e\) factor, while a tree omitting \(e\) has a zero
\(\rho_n(a_e)\) factor.  Hence every monomial on the left and the
common product on the right vanish.  Finally, changing strictness at
\(a_e=0\) or \(a_e=n\) changes only the stated affine null set.
\end{proof}

\begin{proposition}[Canonical law]\label{prop:canonical}
Adjoin an independent uniform root and use its shortest-path
potential.  After the root and tree are forgotten, the induced law
\(\mathbb P_n^{\mathrm{can}}\) of \((D,U)\) has density
\begin{equation}\label{eq:canonical-density-short}
 \frac{n!}{n^{2n-1}}\,
 \1_{\{0\le D_i\le n\ \forall i\}}
 \prod_{i\ne j}\rho_n(D_i+U_j-U_i)\,\mathfrak T(w)
\end{equation}
with respect to \(\dd D\,\dd U\).  Its unnormalized partition
function is
\begin{equation}\label{eq:Zcan}
 Z_n^{\rm can}=\frac{n^{2n-2}}{(n-1)!}.
\end{equation}
The translation \(D_i=A_i+U_i\) has unit Jacobian and gives the
equivalent density
\[
 \frac{n!}{n^{2n-1}}\,
 \1_{\{0\le A_i+U_i\le n\}}
 \prod_{i\ne j}\rho_n(A_i+U_j)\,
 \mathfrak T\bigl((\omega_n(A_i+U_j)\1_{i\ne j})_{ij}\bigr).
\]
Moreover the identities \eqref{eq:minplus} hold exactly.
\end{proposition}

\begin{proof}
Let \(\Sigma(Y)\) be the almost surely unique optimizing
permutation and define the aligned matrix by
\(\widetilde Y_{ij}=Y_{i,\Sigma(Y)(j)}\).  For each fixed
\(\sigma\in S_n\), column permutation maps the optimizer cell with
minimizer \(\sigma\) bijectively and measure-preservingly onto the
identity cell.  Although \(\Sigma(Y)\) is \(Y\)-measurable, the cells
are disjoint outside the optimizer-tie null set; hence, for every
Borel subset \(B\) of the identity cell,
\[
 \Pp\{\widetilde Y\in B\}
 =\sum_{\sigma\in S_n}\operatorname{Leb}(B)
 =n!\operatorname{Leb}(B).
\]
Thus the aligned matrix has density \(n!\) on that cell, with no
independence claim about the selected permutation.  Conditional on
root \(o\),
\cref{lem:charts} gives a unit-Jacobian chart.  Scaling \(d\) gives
\(n^{-n}\dd D\), scaling \(q_{-o}\) gives
\(n^{-(n-1)}\dd V_{-o}\), and centering \(V=nq\) gives
\(\dd V_{-o}=n\dd U\).  Multiplication by the root mass \(1/n\)
therefore gives \(n!/n^{2n-1}\).

A nontree entry integrates over an interval of length
\(\rho_n(a_e)\), while a tree entry is forced to its dual lower bound.
The globally defined factorization in
\cref{lem:boundary-version}, summed over roots and trees, yields
\eqref{eq:canonical-density-short}.  Its
integral is one because it is a pushforward of a probability law,
which gives \eqref{eq:Zcan}.  Finally, dual feasibility and diagonal
equality give both min-plus identities in \eqref{eq:minplus}, with
the diagonal attaining each minimum.  The definition of \(D_i\) gives
\(D_i=A_i+U_i\), and \(\sum_iU_i=0\) gives the cost-sum identity.
\end{proof}

The exact density has now been reduced to explicit row factors and one
residual directed-tree factor.  We next separate these components by
introducing product reference laws.

\section{A product reference law}
\label{sec:reference}

\subsection{The adaptive conditional product}

For fixed \(U\in H_0\) and \(a\in[-U_i,n-U_i]\), define
\[
 F_{i,U}(a)=\prod_{j\ne i}\rho_n(a+U_j),\qquad
 \kappa_{i,U}(a)=\sum_{j\ne i}\omega_n(a+U_j).
\]

\begin{remark}[Interpretation of the construction]
Here \(F_{i,U}(a)\) is the normalized volume of the off-diagonal row entries satisfying all constraints \(nY_{ij}\ge a+U_j\).  The identity \(-F'=(\kappa/n)F\) is  a so-called competing-risks formula: \(\kappa_{i,U}/n\) is the total rate at which a constraint becomes tight, and \(\nu_{i,U}\) is the normalized law of the first tightness level.
\end{remark}

\noindent 
Since \(F'=-\kappa F/n\) almost everywhere,
\begin{equation}\label{eq:Zi-short}
 Z_i(U):=\int_{-U_i}^{n-U_i}\kappa_{i,U}(a)F_{i,U}(a)\,\dd a
 =n\{F_{i,U}(-U_i)-F_{i,U}(n-U_i)\}.
\end{equation}
Whenever \(Z_i>0\), let
\begin{equation}\label{eq:row-kernel-short}
 \nu_{i,U}(\dd a)=Z_i(U)^{-1}
 \1_{[-U_i,n-U_i]}(a)\kappa_{i,U}(a)F_{i,U}(a)\,\dd a.
\end{equation}
When \(Z_i(U)=0\), define instead
\(\nu_{i,U}=\delta_{-U_i}\).  This arbitrary extension makes
\(U\mapsto\nu_{i,U}\) a probability kernel everywhere.  Its precise
value on \(\{Z_i=0\}\) is immaterial under the adaptive law, whose
potential density contains the factor \(\prod_iZ_i\); using the same
extension in both reference laws below keeps their conditional row
kernels literally identical.
The adaptive law is
\begin{equation}
\label{eq:A-law-short}
 \mathbb P_n^{\mathrm{ad}}(\dd U,\dd A)
 =(Z_n^{\rm ad})^{-1}\prod_iZ_i(U)\,\dd U
 \bigotimes_i\nu_{i,U}(\dd A_i).
\end{equation}
Thus conditional independence of the rows under \(\mathbb P_n^{\mathrm{ad}}\) is
a definition, not a property asserted of \(\mathbb P_n^{\mathrm{can}}\).

Let \(m\) be the unique maximizer of \(U\), when it exists, and put
\[
 q_n^{\mathrm{cap}}(U)=\prod_{j\ne m}\frac{U_m-U_j}{n}
\]
on \(\{\osc U<n\}\), and \(q_n^{\mathrm{cap}}=0\) otherwise or when the maximum
is tied.

\begin{lemma}[Potential product]\label{lem:potential-product-short}
\begin{equation}\label{eq:potential-product-short}
 \prod_iZ_i(U)=n^n
 \prod_{i<j}\left(1-\frac{|U_i-U_j|}{n}\right)_+
 \{1-q_n^{\mathrm{cap}}(U)\}.
\end{equation}
\end{lemma}

\begin{proof}
At the lower endpoint,
\(F_{i,U}(-U_i)=\prod_{j\ne i}\rho_n(U_j-U_i)\), and the two
directed factors associated with an unordered pair multiply to
\( (1-|U_i-U_j|/n)_+\).  At the upper endpoint every row except a
unique maximizer has a zero factor; the maximizer contributes exactly
\(q_n^{\mathrm{cap}}\).  For that maximizer \(m\),
\(F_{m,U}(-U_m)=1\), so its endpoint subtraction gives the factor
\(1-q_n^{\mathrm{cap}}\).  If the maximum is tied, every upper-endpoint product is
zero because another maximizer supplies a factor \(\rho_n(n)=0\),
in agreement with the convention \(q_n^{\mathrm{cap}}=0\).  If \(\osc U\ge n\),
both sides vanish: a lower-endpoint pair factor, and hence some
\(Z_i\), is zero.  Substitution in \eqref{eq:Zi-short} now proves the
identity in every sector.
\end{proof}

The factor \(1-q_n^{\mathrm{cap}}\) is an upper-endpoint correction, not part of the
pair interaction.  It appears only for the row carrying the largest
potential: in every other row at least one residual interval has
already closed at the upper endpoint.

\subsection{The solvable gap law}

Define the gap law on \(H_0\) by
\begin{equation}\label{eq:gap-law-short}
 \mathbb P_n^{\rm gap}(\dd U)
 =(Z_n^{\rm gap})^{-1}
 \exp\left\{-\frac1n\sum_{i<j}|U_i-U_j|\right\}\dd U,
\end{equation}
and define
\begin{equation}\label{eq:B-law-short}
 \mathbb P_n^{\mathrm{ref}}(\dd U,\dd A)
 =\mathbb P_n^{\rm gap}(\dd U)\bigotimes_i\nu_{i,U}(\dd A_i).
\end{equation}
On the set where \(Z_i=0\), the preceding kernel convention puts
\(D_i=0\) and \(A_i=-U_i\); we also set every inverse degree below
equal to one.  The
exceptional set satisfies
\[
 \{\exists i:Z_i(U)=0\}\subseteq\{\osc U\ge n\}.
\]
Indeed, if \(\osc U<n\), the first-failure interval in every row has
positive length and positive failure probability.  The adaptive
density contains \(\prod_iZ_i\), so the exceptional set is null under
\(\mathbb P_n^{\mathrm{ad}}\); under the gap law its probability is exponentially
small by the tail estimate below.

\begin{lemma}[Independent gaps and logistic shape]\label{lem:gaps-short}
In the increasing ordering \(U_{(1)}<\cdots<U_{(n)}\),
\begin{equation}\label{eq:spacing-short}
 U_{(k+1)}-U_{(k)}=\frac{nE_k}{k(n-k)},\qquad
 E_k\stackrel{\rm iid}{\sim}{\rm Exp}(1),
\end{equation}
and
\begin{equation}\label{eq:ordered-short}
 U_{(j)}=\sum_{k<j}\frac{E_k}{n-k}-\sum_{k\ge j}\frac{E_k}{k},
 \qquad
 \bar u_{j,n}:=\E U_{(j)}=\mathsf H_{j-1}-\mathsf H_{n-j}.
\end{equation}
Moreover
\[
 Z_n^{\rm gap}=\frac{n^{n-1}}{(n-1)!},\qquad
 \frac1n\sum_i\delta_{U_i}\Longrightarrow
 \ell(u)\dd u,\quad
 \ell(u)=\frac{e^{-u}}{(1+e^{-u})^2},
\]
in \(W_2\), in probability.  If \(R=\osc U\), then
\begin{gather}
 R=O_{\Pp}(\log n),\qquad
 \E\sum_j(U_{(j)}-\bar u_{j,n})^2=2\mathsf H_{n-1},\label{eq:grid-L2-short}\\
 \E\left(\sum_j(U_{(j)}-\bar u_{j,n})^2\right)^2=O(\log^2n),
 \qquad
 \E\left|\frac1n\sum_iU_i^2-\frac{\pi^2}{3}\right|
 =O\!\left(\sqrt{\frac{\log n}{n}}\right).\label{eq:moment-short}
\end{gather}
If \(\varpi_n=n^{-1}\sum_i\delta_{U_i}\) and
\(\bar\varpi_n=n^{-1}\sum_j\delta_{\bar u_{j,n}}\), then
\begin{equation}\label{eq:W1-short}
 \E W_1(\varpi_n,\ell)^2=O(\log n/n),\qquad
 W_1(\bar\varpi_n,\ell)=O(\log n/n).
\end{equation}
For every fixed \(a,b,M\), there are \(A,C<\infty\), depending only
on \(a,b,M\), such that
\begin{equation}\label{eq:weighted-tail-short}
 \E\bigl[n^a(1+R)^b\1_{\{R>A\log n\}}\bigr]\le Cn^{-M}.
\end{equation}
\end{lemma}

\begin{proof}
\medskip\noindent\emph{Exact spacings.}
On one ordering region,
\(\sum_{i<j}|U_i-U_j|=\sum_{k=1}^{n-1}k(n-k)S_k\), where
\(S_k=U_{(k+1)}-U_{(k)}\).  The zero-sum constraint fixes the
location explicitly:
\[
 U_{(1)}=-\frac1n\sum_{k=1}^{n-1}(n-k)S_k,
 \qquad
 U_{(j)}=U_{(1)}+\sum_{k<j}S_k.
\]
Taking \(U_{(1)},\ldots,U_{(n-1)}\) as coordinates on \(H_0\), row
subtraction in this linear map gives the Jacobian \(1/n\).  Here is
the calculation.  The matrix of the map
\((S_1,\ldots,S_{n-1})\mapsto(U_{(1)},\ldots,U_{(n-1)})\) has entries
\[
 M_{jk}=\1_{\{k<j\}}-\frac{n-k}{n}.
\]
Replacing row \(j+1\) by row \(j+1\) minus row \(j\), successively
from bottom to top, turns rows \(2,\ldots,n-1\) into
\(e_1,\ldots,e_{n-2}\).  Expansion in the last column, whose remaining
entry is \(M_{1,n-1}=-1/n\), gives
\(|\det M|=1/n\).  Hence the \(S_k\)'s are
independent exponentials with rates \(k(n-k)/n\).  Summing the \(n!\)
ordering regions proves the partition function and substitution gives
\eqref{eq:ordered-short}.

\medskip\noindent\emph{Logistic location and fluctuation size.}
The harmonic grid differs from the logistic quantiles
\(\log\{j/(n+1-j)\}\) by
\(O(j^{-1}+(n+1-j)^{-1})\).  In \eqref{eq:ordered-short}, the
coefficient of \(E_k-1\) is \(-1/k\) for \(j\le k\) and
\(1/(n-k)\) for \(j>k\).  Squaring and summing gives
\eqref{eq:grid-L2-short}; the standard fourth-moment inequality for a
sum of independent Hilbert-space-valued variables gives its fourth
moment counterpart.  These facts prove the \(W_2\) convergence and
\eqref{eq:moment-short} by Cauchy--Schwarz and Riemann sums.  More
explicitly, if \(\bar\varpi_n=n^{-1}\sum_j\delta_{\bar u_{j,n}}\),
then
\[
 \begin{gathered}
 W_2^2(\bar\varpi_n,\ell)=O(n^{-1}),\qquad
 W_1(\bar\varpi_n,\ell)=O(\log n/n),\\
 \E W_1(\varpi_n,\ell)^2=O(\log n/n).
 \end{gathered}
\]
The first two estimates are one-dimensional quantile Riemann sums;
the last follows by coupling \(U_{(j)}\) with \(\bar u_{j,n}\), using
\eqref{eq:grid-L2-short}, and then applying the triangle inequality.
Finally
\[
 R=\sum_{k=1}^{n-1}\left(\frac1k+\frac1{n-k}\right)E_k,
\]
and, for small fixed \(\theta>0\), independence gives
\[
 \E e^{\theta R}
 =\prod_{k=1}^{n-1}
 \left[1-\theta\left(\frac1k+\frac1{n-k}\right)\right]^{-1}
 \le C_\theta n^{C_\theta}.
\]
Exponential Markov,
with a slightly smaller exponent to absorb \((1+R)^b\), proves
\eqref{eq:weighted-tail-short}.
\end{proof}

\begin{remark}[Why exponentials appear]
Opening the \(k\)-th ordered gap separates \(k\) potentials below it
from \(n-k\) potentials above it.  It therefore increases exactly
\(k(n-k)\) pairwise distances.  The gap-law energy is linear in that
opening, so its Gibbs factor is exponential.  This is the
finite-dimensional analogue of the cavity stability principle that
an exponential intensity is preserved, up to translation and
normalization, when the energies of states receive independent
additive shifts \cite{MezardParisi2003}.  The analogy concerns the
normalized ordered spacings; the potential values themselves have a
logistic empirical law. 
\end{remark}

\subsection{The first likelihood ratio}

Set
\[
 \mathcal R_n(U)=\sum_{i<j}\left[-\log\left(1-
 \frac{|U_i-U_j|}{n}\right)-\frac{|U_i-U_j|}{n}\right],
\]
with \(e^{-\mathcal R_n}=0\) when \(R\ge n\), and put
\(W_n(U)=e^{-\mathcal R_n(U)}(1-q_n^{\mathrm{cap}}(U))\).

\begin{proposition}[The first total-variation comparison]
\label{prop:first-tv-short}
With \(\overline W_n=\E_{\rm gap}W_n\),
\begin{equation}\label{eq:W-limit-short}
 \E_{\rm gap}|W_n-e^{-\zeta(2)}|
 =O\!\left(\sqrt{\frac{\log n}{n}}\right),
 \qquad \overline W_n\to e^{-\zeta(2)}.
\end{equation}
Furthermore
\begin{equation}\label{eq:A-B-RN-short}
 \frac{\dd\mathbb P_n^{\mathrm{ad}}}{\dd\mathbb P_n^{\mathrm{ref}}}
 =\frac{W_n}{\overline W_n},\qquad
 \|\mathbb P_n^{\mathrm{ad}}-\mathbb P_n^{\mathrm{ref}}\|_{\rm TV}
 =O\!\left(\sqrt{\frac{\log n}{n}}\right).
\end{equation}
For all large \(n\), the density in \eqref{eq:A-B-RN-short} is
pointwise bounded by \(2e^{\zeta(2)}\).
\end{proposition}

\begin{proof}
The zero-sum identity
\(\sum_{i<j}(U_i-U_j)^2=n\sum_iU_i^2\) shows that the quadratic
Taylor term in \(\mathcal R_n\) is
\(\frac1{2n}\sum_iU_i^2\).  On \(\{R\le n/2\}\),
\[
 \left|\mathcal R_n-\frac1{2n}\sum_iU_i^2\right|
 \le \frac{2R}{3n^2}\sum_iU_i^2.
\]
Indeed, for \(0\le x\le1/2\),
\(|-\log(1-x)-x-x^2/2|\le 2x^3/3\); summing with
\(|U_i-U_j|\le R\) gives the displayed estimate.  Its expected
right side is \(O(\log n/n)\): use Cauchy--Schwarz together with the
second and fourth moment bounds in \cref{lem:gaps-short}.  Also
\(q_n^{\mathrm{cap}}\le2^{-(n-1)}\) on
\(R\le n/2\), while the complement is exponentially small by
\eqref{eq:weighted-tail-short}.  Since \(x\mapsto e^{-x}\) is
one-Lipschitz on \([0,\infty)\), \eqref{eq:moment-short} proves
\eqref{eq:W-limit-short}.

By \cref{lem:potential-product-short}, the unnormalized marginal
adaptive density is
\[
 n^n\exp\left\{-\frac1n\sum_{i<j}|U_i-U_j|\right\}W_n(U)\,\dd U.
\]
It therefore differs from the gap density precisely by \(W_n\).  The conditional
row kernel is common to both laws, proving the Radon--Nikodym formula.
Finally
\[
 2\|\mathbb P_n^{\mathrm{ad}}-\mathbb P_n^{\mathrm{ref}}\|_{\rm TV}
 =\E_{\rm gap}\left|\frac{W_n}{\overline W_n}-1\right|,
\]
and \(0\le W_n\le1\), giving both remaining assertions.
\end{proof}

The adaptive partition function will be used later:
\begin{equation}\label{eq:Zad-short}
 Z_n^{\rm ad}=n^nZ_n^{\rm gap}\overline W_n
 =\frac{n^{2n-1}}{(n-1)!}\overline W_n.
\end{equation}

\section{One row and the empirical vertex law}
\label{sec:rows-short}

\subsection{The exact first-failure representation}

Write \(D_i=A_i+U_i\) and, for \(0\le d\le n\), put
\[
 S_{i,U}(d)=\prod_{j\ne i}\rho_n(d+U_j-U_i),\qquad
 h_{i,U}(d)=\frac1n\sum_{j\ne i}\omega_n(d+U_j-U_i).
\]

This is the same construction as at the beginning of Section~\ref{sec:reference}.

\begin{lemma}[First failure]\label{lem:first-failure-short}
Conditional on \(U\) under either reference law,
\begin{equation}\label{eq:row-survival-short}
 \Pp(D_i\ge d\mid U)=
 \frac{S_{i,U}(d)-S_{i,U}(n-)}
      {S_{i,U}(0)-S_{i,U}(n-)},\qquad 0\le d<n.
\end{equation}
Equivalently, if independently for \(j\ne i\),
\[
 \Xi_{ij}\sim{\rm Unif}[-U_j,n-U_j],\qquad
 M_i=\min_{j\ne i}\Xi_{ij},
\]
then
\begin{equation}\label{eq:first-failure-short}
 A_i\stackrel d=M_i\mid\{-U_i\le M_i<n-U_i\}.
\end{equation}

If \(J_i\) is the almost surely unique index attaining this minimum,
then, for almost every \(d\) with positive row density,
\begin{equation}\label{eq:first-failure-label-short}
 \Pp(J_i=j\mid U,D_i=d)
 =\frac{\omega_n(d+U_j-U_i)}
 {\sum_{k\ne i}\omega_n(d+U_k-U_i)}.
\end{equation}
The row variables are independent conditional on \(U\).
\end{lemma}

\begin{proof}
Fix \(U\).  For every \(a\in\mathbb R\),
\[
 \{M_i\ge a\}
 =\bigcap_{j\ne i}\{\Xi_{ij}\ge a\}.
\]
Since the \(\Xi_{ij}\)'s are independent and
\(\Pp(\Xi_{ij}\ge a\mid U)=\rho_n(a+U_j)\),
\[
 \Pp(M_i\ge a\mid U)
 =\prod_{j\ne i}\rho_n(a+U_j)
 =F_{i,U}(a).
\]
Thus \(F_{i,U}\) is the survival function of \(M_i\).  At every point
of differentiability,
\[
 -F_{i,U}'(a)
 =\frac1n\kappa_{i,U}(a)F_{i,U}(a).
\]
Moreover, by \eqref{eq:Zi-short},
\[
 \frac{Z_i(U)}n
 =F_{i,U}(-U_i)-F_{i,U}(n-U_i).
\]
It follows that the conditional density of \(M_i\) on
\([-U_i,n-U_i)\) is
\[
 \frac{\1_{[-U_i,n-U_i)}(a)\{-F_{i,U}'(a)\}}
 {F_{i,U}(-U_i)-F_{i,U}(n-U_i)}
 =
 \frac{\1_{[-U_i,n-U_i)}(a)
 \kappa_{i,U}(a)F_{i,U}(a)}
 {Z_i(U)}.
\]
This is precisely the density of \(\nu_{i,U}\), which proves
\eqref{eq:first-failure-short}.

Now \(D_i=A_i+U_i\) and
\[
 S_{i,U}(d)=F_{i,U}(d-U_i).
\]
Therefore, for \(0\le d<n\),
\begin{align*}
 \Pp(D_i\ge d\mid U)
 &=
 \Pp\bigl(M_i\ge d-U_i
   \,\bigm|\,-U_i\le M_i<n-U_i,\ U\bigr)\\
 &=
 \frac{S_{i,U}(d)-S_{i,U}(n-)}
      {S_{i,U}(0)-S_{i,U}(n-)},
\end{align*}
which proves \eqref{eq:row-survival-short}.  Logarithmic
differentiation also gives
\[
 -S_{i,U}'(d)=h_{i,U}(d)S_{i,U}(d)
\]
almost everywhere.

At level \(d\), the hazard contributed by the constraint indexed by
\(j\) is
\[
 \frac{\1_{\{0\le d+U_j-U_i<n\}}}
      {n-d-U_j+U_i}
 =\frac1n\omega_n(d+U_j-U_i).
\]
The density that \(j\) causes the first failure is this hazard
multiplied by the common survival factor.  Conditional on failure at
level \(d\), the survival factor and the conditioning normalization
cancel.  Dividing the \(j\)-th hazard by the sum of all hazards proves
\eqref{eq:first-failure-label-short}.

Finally, under both reference laws the conditional distribution of
\(A\) given \(U\) is the product kernel
\(\bigotimes_i\nu_{i,U}\).  Hence the row variables are independent
conditional on \(U\).
\end{proof}

\begin{lemma}[Uniform row estimates]\label{lem:row-bounds-short}
If \(R=\osc U\le(n-2)/4\), then uniformly in \(i\),
\[
 \left\|\frac{\dd\mathcal L(D_i\mid U)}{\dd d}\right\|_\infty\le C,
 \qquad c\le\Var(D_i\mid U)\le C(R+1)^2,
\]
and for fixed \(p\ge1\),
\begin{equation}\label{eq:row-moment-short}
 \E(|D_i-\E(D_i\mid U)|^p\mid U)\le C_p(R+1)^p,
 \qquad
 \Pp(D_i\ge R+s\mid U)\le Ce^{-s/2}.
\end{equation}
In particular, for every \(M\), some \(B_M\) satisfies
\begin{equation}\label{eq:edge-event-short}
 \Pp_{\mathbb P_n^{\mathrm{ref}}}\{R+\max_iD_i>B_M\log n\}\le C_Mn^{-M}.
\end{equation}
\end{lemma}

\begin{proof}
Write
\[
 b_{ij}(d)=
 \frac{\1_{\{0\le d+U_j-U_i<n\}}}
      {n-d-U_j+U_i},
 \qquad
 h_{i,U}(d)=\sum_{j\ne i}b_{ij}(d).
\]
If \(d\le R\), every nonzero term has denominator at least
\[
 n-d-U_j+U_i\ge n-2R\ge\frac{n+2}{2}.
\]
Consequently,
\[
 h_{i,U}(d)\le\frac{2(n-1)}{n+2}<2.
\]

Now define
\[
 \tau_i=\min_{j\ne i}\{n+U_i-U_j\}.
\]
This is the first value of \(d\) at which one of the factors in
\(S_{i,U}(d)\) becomes zero.  Since
\[
 \tau_i\ge n-R>R,
\]
the survival is positive at \(R\).  
For \(R\le d<\tau_i\), one has
\[
 0\le d+U_j-U_i<n,\qquad j\ne i,
\]
and therefore
\[
 \rho_n(d+U_j-U_i)
 =1-\frac{d+U_j-U_i}{n}.
\]
Thus all \(n-1\) constraints contribute to the hazard, and each
contribution is at least \(1/n\):
\[
 h_{i,U}(d)\ge\frac{n-1}{n}\ge\frac12.
\]
For \(d\ge\tau_i\), at least one factor in the product defining
\(S_{i,U}(d)\) is zero, so \(S_{i,U}(d)=0\).  It follows, in both
cases, that
\[
 \frac{S_{i,U}(R+s)}{S_{i,U}(R)}
 \le e^{-s/2},\qquad s\ge0.
\]
Taking \(s=n-R\) and using \(S_{i,U}(R)\le S_{i,U}(0)\) gives
\[
 \frac{S_{i,U}(n-)}{S_{i,U}(0)}
 \le e^{-(n-R)/2}.
\]
Since \(R\le(n-2)/4\),
\[
 \frac{n-R}{2}\ge\frac{3n+2}{8}\ge1,
\]
and therefore
\[
 S_{i,U}(0)-S_{i,U}(n-)
 \ge(1-e^{-1})S_{i,U}(0).
\]

Using the exact survival formula
\eqref{eq:row-survival-short}, we obtain
\[
 \Pp(D_i\ge R+s\mid U)
 \le\frac{e^{-s/2}}{1-e^{-1}},\qquad s\ge0.
\]
Together with the trivial bound \(\Pp(D_i\ge d\mid U)\le1\) for
\(d\le R\), this yields
\[
 \E(D_i^p\mid U)
 =p\int_0^\infty d^{p-1}\Pp(D_i>d\mid U)\,\dd d
 \le C_p(R+1)^p.
\]
The inequality
\[
 |D_i-\E(D_i\mid U)|^p
 \le2^{p-1}\{D_i^p+\E(D_i\mid U)^p\}
\]
then gives the centered moment bound in
\eqref{eq:row-moment-short}, including the required upper variance
bound.

It remains to bound the row density.  Put
\[
 g_{i,U}(d)=h_{i,U}(d)S_{i,U}(d).
\]
Before \(R\), the preceding estimates give
\[
 g_{i,U}(d)\le2S_{i,U}(0).
\]
On any interval where the active set is fixed,
\[
 h_{i,U}'(d)=\sum_{j\ne i}b_{ij}(d)^2
 \le h_{i,U}(d)^2.
\]
Since \(S_{i,U}'=-h_{i,U}S_{i,U}\),
\[
 g_{i,U}'(d)
 =S_{i,U}(d)
   \{h_{i,U}'(d)-h_{i,U}(d)^2\}
 \le0.
\]
After \(R\), no new lower threshold can enter, so \(g_{i,U}\) is
nonincreasing until \(\tau_i\).  At \(\tau_i\), the survival becomes
zero, and \(g_{i,U}\) is zero thereafter.  Thus
\[
 g_{i,U}(d)\le2S_{i,U}(0)
\]
throughout the support.  Dividing by
\[
 S_{i,U}(0)-S_{i,U}(n-)
 \ge(1-e^{-1})S_{i,U}(0)
\]
shows that the normalized row density is bounded by
\[
 K_0:=\frac2{1-e^{-1}}<3.2.
\]

An interval of radius \(1/(4K_0)\) has probability at most \(1/2\)
under any density bounded by \(K_0\).  Applying this to the interval
centered at \(\E(D_i\mid U)\) gives
\[
 \Var(D_i\mid U)
 \ge\frac1{32K_0^2}>\frac1{330},
\]
which proves the lower variance bound.

Finally, choose \(A\) so that
\[
 \Pp_{\mathbb P_n^{\mathrm{ref}}}\{R>A\log n\}
 =O(n^{-M-2}).
\]
On \(\{R\le A\log n\}\), the row tail bound and a union bound give
\[
 \Pp\{\max_iD_i>R+B\log n\mid U\}
 \le Cn^{1-B/2}.
\]
Choosing \(B>2(M+1)\) and then enlarging the constant multiplying
\(\log n\) proves \eqref{eq:edge-event-short}.
\end{proof}

\subsection{The empirical vertex type}

Let
\[
 \varpi_n=\frac1n\sum_i\delta_{U_i},\qquad
 L(u)=\frac1{1+e^{-u}}, \qquad \ell(u)=\frac{e^{-u}}{(1+e^{-u})^2}.
\]

\begin{lemma}[Uniform local row limit]\label{lem:compact-row-short}
For every fixed \(K<\infty\), under \(\mathbb P_n^{\mathrm{ref}}\),
\begin{equation}\label{eq:compact-row-short}
 \max_{i:|U_i|\le K}\ \sup_{0\le d\le K}
 \left|\Pp(D_i\ge d\mid U)-\frac{L(U_i-d)}{L(U_i)}\right|
 \longrightarrow0
\end{equation}
in probability.
\end{lemma}

\begin{proof}
On \(R\le A\log n\), uniformly in the displayed indices,
\[
 \log\frac{S_{i,U}(d)}{S_{i,U}(0)}
 =-\frac1n\sum_{j\ne i}
 \{(d+U_j-U_i)_+-(U_j-U_i)_+\}
 +O\!\left(\frac{(R+K)^2}{n}\right).
\]
The difference of hinge functions
$x \mapsto (d+x)_+ - x_+$ is one-Lipschitz.  Hence,
 \begin{align*}
 \sup_{\substack{1\le i\le n,\ |U_i|\le K\\0\le d\le K}}
 \left|
 \frac1n\sum_{j\ne i}
 \bigl\{(d+U_j-U_i)_+-(U_j-U_i)_+\bigr\}
 -
 \int_{\mathbb R}
 \bigl\{(d+v-U_i)_+-(v-U_i)_+\bigr\}
 \ell(v)\,\dd v
 \right|
 \\
\qquad \le W_1(\varpi_n,\ell)+\frac Kn
\end{align*}
by definition of $W_1$, the 1-Wasserstein distance.

By \cref{lem:gaps-short} this tends to zero in probability, while
the logistic integral equals
\(\log L(U_i)-\log L(U_i-d)\).  Exponentiation is uniform because
\(L(U_i)\ge L(-K)>0\).

The upper cap affects at most the maximum row.  Writing
\(r_i(d)=S_{i,U}(d)/S_{i,U}(0)\), its exact survival is
\((r_i(d)-q_n^{\mathrm{cap}})/(1-q_n^{\mathrm{cap}})\), whose discrepancy from
\(r_i(d)\) is at most
\(q_n^{\mathrm{cap}}/(1-q_n^{\mathrm{cap}})\).  On the same event,
\(q_n^{\mathrm{cap}}\le(A\log n/n)^{n-1}\).  Finally let \(A\) be large and
use \eqref{eq:weighted-tail-short}.
\end{proof}

We shall also use the following deterministic fact.  
\begin{lemma}[Weak convergence gives uniform cdf convergence]
\label{lem:cdf-uniform-short}
If probability measures \(\mu_n\) on \(\R\) converge weakly to a law
with continuous cdf \(F\), then their cdfs \(F_n\) satisfy
\(\sup_t|F_n(t)-F(t)|\to0\).  The same conclusion holds in
probability for random \(\mu_n\) converging weakly in probability.
\end{lemma}

\begin{proof}
Given \(\varepsilon>0\), choose finitely many continuity points
\(-\infty=t_0<t_1<\cdots<t_m=\infty\) so that
\(F(t_{r+1})-F(t_r)\le\varepsilon\).  Monotonicity traps both cdfs
between their values at adjacent grid points, so convergence on this
finite grid gives \(\sup_t|F_n(t)-F(t)|\le3\varepsilon\) eventually.
For random measures, apply the deterministic assertion along every
subsequence on which weak convergence holds almost surely; the usual
subsequence criterion yields convergence in probability.
\end{proof}

\begin{lemma}[Empirical types]\label{lem:types-short}
Under \(\mathbb P_n^{\mathrm{ref}}\), the empirical law of \((A_i,U_i)\) converges
weakly in probability to
\begin{equation}\label{eq:mustar-short}
 \mu_*(\dd a\,\dd u)=
 \1_{\{a+u\ge0\}}\frac{\ell(a)\ell(u)}{L(u)}\,\dd a\,\dd u.
\end{equation}
Moreover, if
\[
 N_i(d)=\#\{j\ne i:U_j\ge U_i-d\},\qquad
 N_i=N_i(D_i)=\#\{j\ne i:A_i+U_j\ge0\},
\]
then
\begin{equation}\label{eq:degree-uniform-short}
 \sup_i\left|\frac{N_i+1}{n}-L(A_i)\right|\longrightarrow0
\end{equation}
in probability.  The empirical law of \(N_i/n\) therefore converges
to the law on \((0,1)\) with density
\begin{equation}\label{eq:rho-short}
 \rho(x)=-\log(1-x).
\end{equation}
\end{lemma}

\begin{proof}
To establish tightness, we combine second-moment control of the
potential field with an averaged conditional tail bound for \(D_i\).
Put
\[
 V_n=\frac1n\sum_iU_i^2.
\]
  On
\(\{V_n\le K^2/64,R<n\}\), at least \(3n/4\) targets have
\(U_j\ge-K/4\).  Indeed,
\[
 \#\{j:U_j<-K/4\}
 \le\frac{16}{K^2}\sum_jU_j^2
 =\frac{16nV_n}{K^2}\le\frac n4.
\]
We also record the elementary product inequality used next.  For
\(x\in(-n,n)\), set
\[
 a=\frac{(K+x)_+}{n},\qquad b=\frac{x_+}{n}.
\]
If the numerator residual length is nonzero, then
\(0\le b\le a<1\), and
\[
 \frac{1-a}{1-b}
 \le1-(a-b)\le e^{-(a-b)};
\]
the first inequality follows after multiplication by \(1-b\) from
\(b(a-b)\ge0\).  If the numerator is zero, the desired bound is
automatic.  Multiplying this inequality over the row explains the
direction of the following exponential bound.  Thus, if
\(U_i\le K/4\), the product survival at \(K\) for a nonmaximum row
satisfies
\[
 \frac{S_{i,U}(K)}{S_{i,U}(0)}
 \le\exp\left\{-\frac1n\sum_{j\ne i}
 \bigl[(K+U_j-U_i)_+-(U_j-U_i)_+\bigr]\right\}
 \le e^{-K/4}.
\]
For the last inequality, after omitting \(j=i\), at least \(n/2\)
of the preceding \(3n/4\) targets remain (for \(n\ge4\)); each has
\(U_j-U_i\ge-K/2\), so its hinge increment is at least \(K/2\).
The finitely many smaller \(n\) are harmless.
The possible maximum row is isolated as the first term below.  Hence
\[
 \frac1n\sum_i\Pp(D_i>K\mid U)
 \le \frac1n+e^{-K/4}
 +\frac1n\sum_i\1_{\{U_i>K/4\}}
 +\1_{\{V_n>K^2/64\}}+\1_{\{R\ge n\}}.
\]
The expectation of the right side tends to zero as first
\(n\to\infty\) and then \(K\to\infty\).  Thus the compact convergence
of \cref{lem:compact-row-short} extends to bounded continuous tests.
Conditional independence makes
the conditional variance of each empirical average \(O(1/n)\), and
differentiating the limiting survival gives
\[
 \Pp(D\in\dd d\mid U=u)=\frac{\ell(u-d)}{L(u)}\1_{\{d\ge0\}}\dd d.
\]
The change of variables \(a=d-u\), and logistic symmetry, yield
\eqref{eq:mustar-short}.

The full-diagonal active proportion is exactly
\[
 \frac{N_i+1}{n}=\frac1n\sum_j\1_{\{U_j\ge-A_i\}}.
\]
By \cref{lem:gaps-short}, the empirical measures of \(U\) converge
weakly in probability to the logistic law.  The deterministic
\cref{lem:cdf-uniform-short} therefore gives uniform convergence of
their cdfs despite the dependence among the \(U_i\)'s.  Since the gap
law has no ties and \(1-L(-a)=L(a)\), evaluation at the random points
\(-A_i\) proves
\eqref{eq:degree-uniform-short}.  Under \(\mu_*\), the change
\(x=L(a)\) gives the \(A\)-marginal
\[
 \ell(a)\left\{\int_{-a}^{\infty}\frac{\ell(u)}{L(u)}\,\dd u\right\}
 \dd a
 =-\ell(a)\log\{1-L(a)\}\,\dd a,
\]
which becomes \(-\log(1-x)\dd x\).  This proves
\eqref{eq:rho-short}.
\end{proof}

The inverse degree is singular at zero, so weak convergence alone is
insufficient.  The following estimate is the quantitative endpoint
input for the tree factor.

\subsection{Sparse degrees and the conditional row CLT}

Whenever \(Z_i(U)>0\), the factor \(\kappa_{i,U}(a)\) in the
density of \(\nu_{i,U}\) implies that, almost surely,
\[
 0\le A_i+U_j<n
\]
for at least one \(j\ne i\).  Consequently \(N_i\ge1\).  On the
exceptional set where \(Z_i(U)=0\), inverse degrees are interpreted
according to the convention introduced above.

\begin{lemma}[Singular inverse degrees]\label{lem:sparse-short}
For \(1\le K\le n/2\),
\begin{equation}\label{eq:sparse-short}
 \E_{\mathbb P_n^{\mathrm{ref}}}\sum_i
 \frac{\1_{\{1\le N_i\le K\}}}{N_i}\le C\frac Kn,
 \qquad
 \E_{\mathbb P_n^{\mathrm{ref}}}\sum_i\frac1{N_i^2}
 \le C\frac{\log n}{n}.
\end{equation}
Consequently
\begin{equation}\label{eq:inv-degree-limit-short}
 \sum_i\frac1{N_i}\longrightarrow\zeta(2)
 \quad\hbox{in }L^1(\mathbb P_n^{\mathrm{ref}}).
\end{equation}
\end{lemma}

\begin{proof}
Condition on \(U\) and first work on the event \(R\le n/4\).  In the
notation of \cref{lem:first-failure-short}, put
\[
 X_{ij}=\Xi_{ij}+U_i.
\]
Then \(D_i\) has the distribution of
\(\min_{j\ne i}X_{ij}\), conditional on this minimum belonging to
\([0,n)\).

We impose this conditioning in two steps.  First condition on
\(X_{ij}\ge0\) for every \(j\ne i\).  Because this event factors over
\(j\), the variables remain independent.  Each conditional uniform
interval has length at least \(n-R\), so its density is at most
\[
 (n-R)^{-1}\le\frac4{3n}.
\]
It remains to condition the minimum to be less than \(n\).  The event
that all the \(X_{ij}\)'s are at least \(n\) has probability zero
unless \(U_i\) is the largest potential.  In that exceptional row its
probability is at most \(4^{-(n-1)}\).  The second conditioning
therefore increases densities by a factor of at most
\[
 \bigl(1-4^{-(n-1)}\bigr)^{-1},
\]
and
\[
 \frac4{3n}\frac1{1-4^{-(n-1)}}<\frac2n.
\]

At a given level \(d\), only indices satisfying
\[
 U_j\ge U_i-d
\]
can attain the minimum.  There are \(N_i(d)\) such indices.  Summing
their possible density contributions shows that the conditional
density of \(D_i\) at \(d\) is at most \(2N_i(d)/n\).  Hence, for
every nonnegative function \(g\),
\begin{equation}\label{eq:failure-dom-short}
 \E\{g(N_i(D_i))\mid U\}
 \le\frac2n\int_0^n g(N_i(d))N_i(d)\,\dd d.
\end{equation}

We next evaluate sums of the integrals on the right.  Order the
potentials decreasingly:
\[
 V_1>\cdots>V_n,\qquad s_m=V_m-V_{m+1}.
\]
For any nonnegative function \(\Phi\), the change of variables
\(t=U_i-d\), followed by summation over \(i\), gives
\begin{equation}\label{eq:sweep-short}
 \sum_i\int_0^n
 \1_{\{N_i(d)\le n-2\}}\Phi(N_i(d))\,\dd d
 =
 \sum_{m=1}^{n-1}m\Phi(m-1)s_m.
\end{equation}
Indeed, fix \(t\in(V_{m+1},V_m)\).  Exactly \(m\) potentials exceed
\(t\).  For each of the corresponding \(m\) rows, set
\(d=U_i-t\).  The count \(N_i(d)\) includes the other \(m-1\)
potentials above \(t\), but excludes the source label \(i\), and is
therefore equal to \(m-1\).  The set of values of \(t\) for which this
count remains unchanged has length \(s_m\).  Moreover,
\[
 0<d=U_i-t\le R<n,
\]
so the restriction \(0\le d\le n\) does not remove any of these
contributions.  Values of \(t\) below the smallest potential give
\(N_i(d)=n-1\) and are excluded by the indicator on the left.

Apply \eqref{eq:failure-dom-short} with
\[
 g(k)=\frac{\1_{\{1\le k\le K\}}}{k}.
\]
The factor \(N_i(d)\) cancels the denominator.  Using
\eqref{eq:sweep-short} with
\(\Phi(k)=\1_{\{1\le k\le K\}}\) gives
\[
 \E\sum_{m=2}^{K+1}ms_m
 =
 n\sum_{m=2}^{K+1}\frac1{n-m}
 \le CK,
\]
where we used
\[
 \E s_m=\frac{n}{m(n-m)}.
\]
Multiplication by \(2/n\) proves the first estimate in
\eqref{eq:sparse-short}.

For the second estimate, take
\[
 g(k)=\frac{\1_{\{1\le k\le n-2\}}}{k^2}.
\]
Then \(g(k)k=1/k\) for \(1\le k\le n-2\).  Applying
\eqref{eq:sweep-short} with
\[
 \Phi(0)=0,\qquad
 \Phi(k)=\frac1k,\quad 1\le k\le n-2,
\]
and again using the mean of \(s_m\), the factor \(2/n\) in
\eqref{eq:failure-dom-short} gives
\[
 \frac2n\E\sum_{m=2}^{n-1}\frac{m}{m-1}s_m
 =2\sum_{m=2}^{n-1}\frac1{(m-1)(n-m)}
 =
 O\left(\frac{\log n}{n}\right).
\]
The remaining case \(N_i=n-1\) contributes at most
\[
 \frac{n}{(n-1)^2}=O(n^{-1}).
\]
On the complementary event \(R>n/4\), each of the sums under
consideration is at most \(n\), while the probability of this event
is exponentially small.  This proves \eqref{eq:sparse-short}.

Finally, fix \(\eta>0\).  By \cref{lem:types-short},
\[
 \sum_i\frac{\1_{\{N_i>\eta n\}}}{N_i}
 \longrightarrow
 \int_\eta^1\frac{-\log(1-x)}x\,\dd x
\]
in \(L^1\); the left side is bounded by \(1/\eta\).  The first
estimate in \eqref{eq:sparse-short}, with
\(K=\lfloor\eta n\rfloor\), shows that
\[
 \E\sum_i\frac{\1_{\{1\le N_i\le\eta n\}}}{N_i}
 \le C\eta.
\]
Letting \(\eta\downarrow0\) and using
\[
 \int_0^1\frac{-\log(1-x)}x\,\dd x=\zeta(2)
\]
proves \eqref{eq:inv-degree-limit-short}.
\end{proof}

If \(m_{i,n}=\E(D_i\mid U)\) and
\(v_{i,n}=\Var(D_i\mid U)\), the preceding moment estimates give the
conditional row central limit theorem in the precise form needed
later.

\begin{lemma}[Conditional row characteristic function]
\label{lem:conditional-cf-short}
For every fixed \(t\), under \(\mathbb P_n^{\mathrm{ref}}\),
\begin{equation}\label{eq:conditional-cf-short}
 \E\left[\exp\left\{\frac{it}{\sqrt n}
 \sum_i(D_i-m_{i,n})\right\}\middle|U\right]
 =\exp\left\{-\frac{t^2}{2n}\sum_iv_{i,n}+o_{\Pp}(1)\right\}.
\end{equation}
\end{lemma}

\begin{proof}
Conditional on \(U\), the rows are independent.  Taylor expansion of
each centered conditional characteristic function gives a third-order
remainder bounded by
\[
 \frac{C_t}{n^{3/2}}
 \E(|D_i-m_{i,n}|^3\mid U).
\]
The moment bound in \cref{lem:row-bounds-short} is asserted only on
\[
 \mathcal G_n=\{R\le(n-2)/4\}.
\]
This restriction causes no loss: by
\eqref{eq:weighted-tail-short},
\(\Pp(\mathcal G_n^c)=O(n^{-M})\) for every fixed \(M\).
On \(\mathcal G_n\), after summing over \(i\),
\cref{lem:row-bounds-short} bounds the remainder by
\(C_t(R+1)^3/\sqrt n=o_{\Pp}(1)\).  Expanding the logarithm of the
product costs at most \(C_t(R+1)^4/n=o_{\Pp}(1)\).  The quadratic
terms give the exponent in \eqref{eq:conditional-cf-short}.
\end{proof}

\section{The Ferrers tree factor}
\label{sec:ferrers-short}

The adaptive product law \(\mathbb P_n^{\mathrm{ad}}\) accounts for the local constraints in each
row, but it does not yet impose that the corresponding tight edges fit
together into one tree.  This remaining global requirement is carried
by \(\mathfrak T(w)\).  The purpose of this section is to show that,
after the local row weights have been removed, the tree factor
converges to a constant.

The argument has three steps.  First, the matrix-tree theorem writes
the normalized tree factor as a determinant.  Second, after ordering
the potentials, the active entries in every row form a block starting
from the first column.  Replacing their weights by one then produces a
nested zero--one matrix whose determinant can be computed exactly.
Finally, we put back the true weights.  They differ only slightly from
one on the relevant event, so the determinant changes by a negligible
amount.  The resulting constant makes the likelihood ratio between
the canonical and adaptive laws asymptotically flat.

\subsection{A bounded cofactor below the tree likelihood}

At any configuration with positive row weights, recalling the notation
from Section~\ref{sec:RandRoot}, put
\[
 \kappa_i=\sum_jw_{ij},\qquad P_{ij}=\frac{w_{ij}}{\kappa_i},
 \qquad
 \widehat{\mathcal T}_n
 =n\frac{\mathfrak T(w)}{\prod_i\kappa_i}.
\]
Under the adaptive law, \(\kappa_i>0\) almost surely for every \(i\),
so \(P\) is row-stochastic.  The preceding first-failure calculation gives its
probabilistic meaning:
\[
 P_{ij}=\Pp(J_i=j\mid U,D_i).
\]
Thus \(\kappa_i/n\) is the total rate at which some off-diagonal
constraint in row \(i\) becomes tight, and \(P_{ij}\) is the share of
that risk carried by column \(j\).  

It also explains exactly why the normalized tree polynomial is a
Markov-chain quantity.  Fix a root \(o\), and independently choose an
outgoing successor \(J_i\) with law \(P_{i\cdot}\) for every
\(i\ne o\).  For a particular in-arborescence \(T\) toward \(o\),
\[
 \prod_{e\in T}w_e
 =\left(\prod_{i\ne o}\kappa_i\right)
   \prod_{e\in T}P_e.
\]
Consequently \(\tau_o(P)\) is the probability that all these arrows
form an in-arborescence toward \(o\), equivalently that every chosen
path eventually reaches \(o\).  The root has no outgoing tree edge,
which leaves the factor \(1/\kappa_o\); summing over roots gives
\begin{equation}
\label{eq:normalized-tree-risk-short}
 \frac{\mathfrak T(w)}{\prod_i\kappa_i}
 =\sum_o\frac{\tau_o(P)}{\kappa_o},
 \qquad
 \widehat{\mathcal T}_n
 =\sum_o\tau_o(P)\frac n{\kappa_o}.
\end{equation}
In words, \(\kappa_i\) is the total weight of the possible outgoing
boundary edges from each nonroot vertex \(i\).  After these local
weights are factored out, \(\tau_o(P)\) imposes the global requirement
that the selected edges form a single tree directed toward \(o\),
rather than directed cycles or disconnected components.
Thus, \(\widehat{\mathcal T}_n\) is the residual global weight of
the requirement that the locally selected outgoing edges form a
single rooted spanning tree.

For \(M>0\), let \(f_M(x)=M\wedge x^{-1}\), set
\[
 (u_M)_i=f_M(\kappa_i/n),\qquad 1\le i\le n,
\]
and define
\begin{equation}\label{eq:bounded-det-short}
\mathcal K_{n,M}
 =\det\Bigl\{I-P+\tfrac1n\,u_M\mathbf1^{\mathsf T}\Bigr\}.
\end{equation}
Here \(\tfrac1n u_M\mathbf1^{\mathsf T}\) is the rank-one matrix with
entries \((u_M)_i/n\); on the adaptive law \(\kappa_i>0\) for every
\(i\), so \(P\) is row-stochastic and \(u_M\) is well defined.

\begin{lemma}[Cofactor lower bound]\label{lem:cofactor-short}
For every \(M>0\),
\begin{equation}\label{eq:cofactor-bound-short}
 0\le\mathcal K_{n,M}\le\widehat{\mathcal T}_n.
\end{equation}
\end{lemma}

\begin{proof}
Since \(P\) is row-stochastic, \((I-P)\mathbf1=0\) and hence
\(\det(I-P)=0\).  The directed matrix-tree theorem in the
outgoing-edge convention identifies the adjugate as
\[
 \operatorname{adj}(I-P)=\mathbf1\,\bigl(\tau_1(P),\ldots,\tau_n(P)\bigr),
 \qquad\text{i.e.}\qquad
 \operatorname{adj}(I-P)_{io}=\tau_o(P)
\]
for all \(i,o\) \cite{ChaikenKleitman1978}.  For a square matrix \(B\)
and column vectors \(u,v\in\mathbb R^n\), the rank-one determinant
identity reads
\[
 \det\bigl(B+uv^{\mathsf T}\bigr)
 =\det B+v^{\mathsf T}\operatorname{adj}(B)\,u .
\]
Apply it with \(B=I-P\), \(u=u_M\), and \(v^{\mathsf T}=\tfrac1n\mathbf1^{\mathsf T}\).
The adjugate formula gives
\[
 \operatorname{adj}(I-P)\,u_M
 =\Bigl(\sum_o\tau_o(P)f_M(\kappa_o/n)\Bigr)\mathbf1 ,
\]
and \(\tfrac1n\mathbf1^{\mathsf T}\mathbf1=1\), so that
\[
 \mathcal K_{n,M}=\sum_o\tau_o(P)\,f_M(\kappa_o/n).
\]
On the other hand, by \eqref{eq:normalized-tree-risk-short},
\[
 \widehat{\mathcal T}_n=\sum_o\tau_o(P)\,\frac n{\kappa_o}.
\]
The tree weights \(\tau_o(P)\) are nonnegative, being sums of products
of nonnegative transition probabilities, and for every \(o\),
\[
 0\le f_M(\kappa_o/n)\le\frac n{\kappa_o}.
\]
Summing these inequalities over \(o\) proves
\eqref{eq:cofactor-bound-short}.
\end{proof}

\subsection{A six-vertex picture}

Before giving the general argument, consider a configuration with
\(n=6\) and decreasing potentials
\[
 U=\left(\frac54,\frac34,\frac14,-\frac14,-\frac34,-\frac54\right).
\]
Choose
\[
 A=\left(-\frac12,0,0,1,\frac32,\frac32\right),
 \qquad
 D=A+U
 =\left(\frac34,\frac34,\frac14,
          \frac34,\frac34,\frac14\right).
\]
Every nonnegative residual \(A_i+U_j\) is smaller than \(6\), so the
upper cutoff removes no entry in this example.
The condition \(A_i+U_j\ge0\), with the diagonal entry removed, gives
the active-edge matrix
\[
 (c_{ij})_{1\le i,j\le6}
 =
 \begin{pmatrix}
 0&1&0&0&0&0\\
 1&0&1&0&0&0\\
 1&1&0&0&0&0\\
 1&1&1&0&1&0\\
 1&1&1&1&0&1\\
 1&1&1&1&1&0
 \end{pmatrix}.
\]
Its row degrees and the corresponding prefix lengths are
\[
 (N_i)=(1,2,2,4,5,5),
 \qquad
 (k_i)=(N_i+1)=(2,3,3,5,6,6).
\]
Adding back the diagonal turns row \(i\) into a solid block of ones in
columns \(1,\ldots,k_i\).  Thus
\[
 P^{(0)}_{ij}=\frac{c_{ij}}{N_i},\qquad
 Q_{ij}=\frac{\1_{\{j\le k_i\}}}{k_i},
\]
and
\[
 I-P^{(0)}
 =\diag\left(\frac{k_i}{k_i-1}\right)(I-Q).
\]
This is the basic reason for introducing the restored diagonal: the
rows of \(Q\) are nested prefixes.  After the elementary row operations
used below, the determinant becomes upper triangular.  In this example,
the exact product formula is
\[
 \mathcal T_6^{(0)}
 =\prod_{i=1}^6\left(1+\frac1{N_i}\right)
 =2\left(\frac32\right)^2\frac54
   \left(\frac65\right)^2
 =\frac{81}{10}.
\]
For the true matrix, every active entry \(1\) above is replaced by
\(1+\varepsilon_{ij}\).  The remainder of the section shows that these
small changes do not alter the determinant asymptotically.

\subsection{The nested zero--one support}

Fix a tail index \(p\) and let \(B_p\) be as in
\eqref{eq:edge-event-short}.  All Ferrers identities below are
asserted on
\[
 \mathcal E_{n,p}
 =\{\,\osc U+\max_iD_i\le B_p\log n\,\}.
\]
If \(A_i+U_j\ge0\), then
\[
 0\le A_i+U_j=D_i+U_j-U_i
 \le D_i+\osc U\le B_p\log n<n/2
\]
for all large \(n\).  Hence on \(\mathcal E_{n,p}\) the support of
\(w\) is exactly
\[
 c_{ij}=\1_{\{i\ne j\}}\1_{\{A_i+U_j\ge0\}},\qquad
 N_i=\sum_jc_{ij},\qquad P^{(0)}_{ij}=c_{ij}/N_i,
\]
and this identification is pointwise.  On an active edge, write
\[
 a_{ij}=A_i+U_j=D_i+U_j-U_i.
\]
Since \(\omega_n(a)=(1-a/n)^{-1}\) for \(0\le a<n\), we have
\[
 w_{ij}=c_{ij}(1+\varepsilon_{ij}),\qquad
 \varepsilon_{ij}=\frac{a_{ij}}{n-a_{ij}}
\]
on active edges, and we set \(\varepsilon_{ij}=0\) otherwise.  Thus,
pointwise on \(\mathcal E_{n,p}\),
\[
 0\le\varepsilon_{ij}
 \le\frac{B_p\log n}{n-B_p\log n}
 \le\frac{2B_p\log n}{n}
\]
for every \(i,j\) and all sufficiently large \(n\).  In particular,
\(\kappa_i\ge N_i\) for every \(i\).
On \(\mathcal E_{n,p}^{\,c}\) the auxiliary objects may be extended
arbitrarily.  Since the convergence statements below that use these
objects are in probability under \(\mathbb P_n^{\mathrm{ref}}\),
\eqref{eq:edge-event-short} makes this extension immaterial.  Order the potentials
decreasingly by applying the same permutation to rows and columns;
this preserves both the aligned diagonal and every determinant.
Restoring the diagonal, the active set in row \(i\) is
a prefix
\begin{equation}\label{eq:prefix-short}
 \{j:A_i+U_j\ge0\}=\{1,\ldots,k_i\},
 \qquad k_i=N_i+1\ge i.
\end{equation}
Indeed, \(A_i+U_i=D_i\ge0\), so the restored diagonal lies in the
prefix and \(i\le k_i\).

\begin{lemma}[Ferrers product]\label{lem:ferrers-product-short}
Let
\[
 \mathcal T_n^{(0)}=
 \det\left\{I-P^{(0)}+ \frac1n\left(\frac n{N_i}\right)_i
\mathbf1^{\mathsf T}   \right\}.
\]
Then
\begin{equation}\label{eq:ferrers-product-short}
 \mathcal T_n^{(0)}=\prod_i\left(1+\frac1{N_i}\right)
 \longrightarrow e^{\zeta(2)}
\end{equation}
in probability under \(\mathbb P_n^{\mathrm{ref}}\).
\end{lemma}

\begin{proof}
Let \(Q_{ij}=\1_{\{j\le k_i\}}/k_i\) and
\(C=\diag(k_i/(k_i-1))\).  Direct inspection gives
\(I-P^{(0)}=C(I-Q)\): restoring the diagonal changes a uniform row
of length \(k_i\) into the same row with its diagonal entry removed
and the remaining \(k_i-1=N_i\) entries renormalized.  Since
\(C^{-1}(n/N_i)_i=(n/k_i)_i\), factoring \(C\) and then multiplying
the remaining determinant on the left by \(K=\diag(k_i)\) reduces it to
\[
 B=K-\mathsf F+\mathbf1\mathbf1^{\mathsf T},\qquad
 \mathsf F_{ij}=\1_{\{j\le k_i\}}.
\]
Its diagonal is \(k_i\); an off-diagonal entry can be nonzero only
when \(j>k_i\ge i\).  Thus \(B\) is upper triangular and
\(\det B=\prod_ik_i\).  Keeping track of the two diagonal factors
gives
\[
 \mathcal T_n^{(0)}
 =\prod_i\frac{k_i}{k_i-1}
 =\prod_i\left(1+\frac1{N_i}\right),
\]
which is the exact identity.  By
\cref{lem:sparse-short},
\[
 \sum_iN_i^{-1}\to\zeta(2),\qquad
 \sum_iN_i^{-2}\to0
\]
in probability.  The inequality
\(0\le x-\log(1+x)\le x^2/2\) proves the limit.
\end{proof}

\subsection{Truncating the root weight and restoring the true weights}

The bounded-root version of \(\mathcal T_n^{(0)}\) is
\[
 \mathcal K_{n,M}^{(0)}=
 \det\Bigl\{I-P^{(0)}+
 \tfrac1n(f_M(N_i/n))_i\mathbf1^{\mathsf T}\Bigr\}.
\]

\begin{lemma}[Root truncation and perturbation]\label{lem:root-perturb-short}
There are weights \(\omega_i^{\rm root}\ge0\), summing to one and
satisfying \(\omega_i^{\rm root}\le1/(N_i+1)\), such that
\begin{equation}\label{eq:root-mixture-short}
 \frac{\mathcal K_{n,M}^{(0)}}{\mathcal T_n^{(0)}}
 =\sum_i\omega_i^{\rm root}\min\left\{\frac{MN_i}{n},1\right\}.
\end{equation}
Consequently
\begin{equation}\label{eq:root-loss-short}
 0\le1-\frac{\mathcal K_{n,M}^{(0)}}{\mathcal T_n^{(0)}}
 \le\sum_{N_i<n/M}\frac1{N_i}.
\end{equation}
For fixed \(M\),
\begin{equation}\label{eq:det-perturb-short}
 \mathcal K_{n,M}-\mathcal K_{n,M}^{(0)}\longrightarrow0
\end{equation}
in probability under \(\mathbb P_n^{\mathrm{ref}}\).
\end{lemma}

\begin{proof}
\medskip\noindent\emph{Root truncation.}
With \(B\) as above, the rank-one determinant formula gives
\eqref{eq:root-mixture-short} with
\((\omega^{\rm root})^{\mathsf T}=\mathbf1^{\mathsf T}B^{-1}\).
Indeed, if
\[
 \widetilde L=K-\mathsf F=\diag(N_i)(I-P^{(0)}),\qquad
 B=\widetilde L+\mathbf1\mathbf1^{\mathsf T},
\]
then \(\widetilde L\mathbf1=0\), hence \(B\mathbf1=n\mathbf1\) and
\[
 (\omega^{\rm root})^{\mathsf T}\mathbf1
 =\mathbf1^{\mathsf T}B^{-1}\mathbf1=1.
\]
The identity
\(\det(\widetilde L+z\mathbf1^{\mathsf T})
=\mathbf1^{\mathsf T}\operatorname{adj}(\widetilde L)z\), the matrix-tree
theorem \cite{ChaikenKleitman1978}, and
\(\det B=(\prod_iN_i)\mathcal T_n^{(0)}\) give
\[
 \omega_o^{\rm root}
 =\frac{\tau_o(P^{(0)})\,n/N_o}{\mathcal T_n^{(0)}}\ge0
 \qquad(1\le o\le n).
\]
The \(j\)-th column of
\((\omega^{\rm root})^{\mathsf T}B=\mathbf1^{\mathsf T}\) reads
\[
 k_j\omega_j^{\rm root}
 +\sum_{i:k_i<j}\omega_i^{\rm root}=1,
\]
so \(\omega_j^{\rm root}\le1/k_j\).  This proves
\eqref{eq:root-loss-short}.  For \(M\ge2\), its expected right side is
\(O(1/M)\) by \eqref{eq:sparse-short}.

\medskip\noindent\emph{Same-dimension perturbation.}
On \(\mathcal E_{n,p}\), put
\[
 \bar\varepsilon_i=\frac1{N_i}\sum_jc_{ij}\varepsilon_{ij}.
\]
Then \(\kappa_i=N_i(1+\bar\varepsilon_i)\), and
\[
 P_{ij}-P^{(0)}_{ij}
 =\frac{c_{ij}}{N_i}
   \frac{\varepsilon_{ij}-\bar\varepsilon_i}
        {1+\bar\varepsilon_i}.
\]
Since \(0\le\bar\varepsilon_i\le\max_{k,l}\varepsilon_{kl}\),
\[
 |\varepsilon_{ij}-\bar\varepsilon_i|
 \le\max_{k,l}\varepsilon_{kl}.
\]
Consequently,
\[
 \|P-P^{(0)}\|_{\rm F}^2
 \le\left(\max_{k,l}\varepsilon_{kl}\right)^2
       \sum_iN_i^{-1}=o_{\Pp}(1),
\]
where we used the deterministic bound above and
\eqref{eq:inv-degree-limit-short}.  The same bound gives, pointwise on
\(\mathcal E_{n,p}\),
\[
 0\le\kappa_i-N_i=N_i\bar\varepsilon_i
 \le 2B_p\log n,
\]
and hence
\(\max_i|\kappa_i-N_i|/n=O(\log n/n)\).  Since \(f_M\) is
\(M^2\)-Lipschitz, if
\(u_M^{(0)}=(f_M(N_i/n))_i\), then
\[
 \left\|\tfrac1n(u_M-u_M^{(0)})\mathbf1^{\mathsf T}\right\|_{\rm F}^2
 =\frac1n\sum_i|u_{M,i}-u_{M,i}^{(0)}|^2=o_{\Pp}(1).
\]
For each fixed \(M\), the two matrices have tight Frobenius norms because
\(\|P^{(0)}\|_{\rm F}^2=\sum_iN_i^{-1}\), while the root vectors are
bounded by \(M\).  Thus the exponential factor below is
\(O_{\Pp,M}(1)\); its bound is allowed to depend on \(M\).  The
dimension-free estimate
\[
 |\det{}_2(I-S)-\det{}_2(I-T)|
 \le\|S-T\|_{\rm F}
 \exp\{C(1+\|S\|_{\rm F}+\|T\|_{\rm F})^2\}
\]
therefore applies on events of arbitrarily high probability
\cite[Theorem~9.2(c)]{Simon2005}.  To check the trace correction, put
\(S=P-\tfrac1n u_M\mathbf1^{\mathsf T}\), and for the comparison matrix use
\(S^{(0)}=P^{(0)}-\tfrac1n u_M^{(0)}\mathbf1^{\mathsf T}\).  Since \(P\) has
zero diagonal,
\[
 \operatorname{tr}S=-\frac1n\sum_i u_{M,i},\qquad
 \det\Bigl(I-P+\tfrac1n u_M\mathbf1^{\mathsf T}\Bigr)
 =\exp\left\{\frac1n\sum_i u_{M,i}\right\}\det{}_2(I-S).
\]
The \(L^2\) convergence of the root vectors implies
\(\left|n^{-1}\sum_i(u_{M,i}-u_{M,i}^{(0)})\right|=o_{\Pp}(1)\)
by Cauchy--Schwarz.
Thus, for each fixed \(M\), both the regularized determinant and its scalar
trace factor are continuous under the preceding Frobenius convergence.  This proves
\eqref{eq:det-perturb-short}.  By \eqref{eq:edge-event-short}, the
complement of \(\mathcal E_{n,p}\) does not affect this convergence in
probability.
\end{proof}

\subsection{The normalization squeeze and the second comparison}

\begin{proposition}[The second total-variation comparison]
\label{prop:second-tv-short}
\begin{equation}\label{eq:tree-L1-short}
 \widehat{\mathcal T}_n\longrightarrow e^{\zeta(2)}
 \quad\hbox{in }L^1(\mathbb P_n^{\mathrm{ad}}),
\end{equation}
and
\begin{equation}\label{eq:second-tv-short}
 \|\mathbb P_n^{\mathrm{can}}-\mathbb P_n^{\mathrm{ad}}\|_{\rm TV}\longrightarrow0.
\end{equation}
Hence \(\|\mathbb P_n^{\mathrm{can}}-\mathbb P_n^{\mathrm{ref}}\|_{\rm TV}\to0\).
\end{proposition}

\begin{proof}
By \cref{lem:cofactor-short,lem:ferrers-product-short,lem:root-perturb-short},
for fixed \(M\),
\[
 \widehat{\mathcal T}_n\ge\mathcal K_{n,M}
 \ge\mathcal T_n^{(0)}
 \left(1-\sum_{N_i<n/M}N_i^{-1}\right)+o_{\Pp}(1).
\]
The expected truncated sum is \(O(1/M)\).  First let \(n\to\infty\),
then \(M\to\infty\), to obtain under \(\mathbb P_n^{\mathrm{ref}}\)
\begin{equation}\label{eq:tree-lower-short}
 \Pp\{\widehat{\mathcal T}_n<e^{\zeta(2)}-\varepsilon\}\to0.
\end{equation}
The first total-variation comparison transfers this one-sided bound
to \(\mathbb P_n^{\mathrm{ad}}\).  Notice the order: the empirical logistic type
and the Ferrers lower bound were proved under the explicit reference
law before any use of the conclusion sought here.

Let \(\mathbb P_n^{\mathrm{can},+}\) be the unnormalized restriction of the
canonical law to the sector \(\kappa_i>0\) for every \(i\), and let
\(p_n^+\) be its mass.  These are the same degrees that appear in the
adaptive density.  Indeed, the potential density of
\(\mathbb P_n^{\mathrm{ad}}\) contains \(\prod_iZ_i\), so
\(Z_i>0\) for every \(i\), almost surely.  Conditional on such a
potential, \(\nu_{i,U}\) has density proportional to
\(\kappa_{i,U}F_{i,U}\); it therefore assigns probability one to
\(\kappa_i=\kappa_{i,U}(A_i)>0\).  Thus the adaptive law is supported
entirely on the positive-degree sector, while the missing canonical
mass is exactly \(1-p_n^+\).  Comparing \eqref{eq:canonical-density-short},
\eqref{eq:A-law-short}, \eqref{eq:Zcan}, and \eqref{eq:Zad-short}
gives the exact identity
\begin{equation}\label{eq:Q-A-RN-short}
 \frac{\dd\mathbb P_n^{\mathrm{can},+}}{\dd\mathbb P_n^{\mathrm{ad}}}
 =\overline W_n\widehat{\mathcal T}_n.
\end{equation}
No additional canonical mass is hidden on the set
\(\{\osc U\ge n\}\), where the adaptive potential density vanishes.
Indeed, choose \(i\) with minimal potential and \(j\) with maximal
potential.  Since \(D_i\ge0\),
\[
 a_{ij}=D_i+U_j-U_i\ge U_j-U_i=\osc U\ge n.
\]
Thus \(\rho_n(a_{ij})=0\); by the pointwise boundary convention in
\cref{lem:boundary-version}, every canonical tree monomial vanishes.
Hence the canonical density itself is zero on this set.
Therefore
\[
 \E_{\mathbb P_n^{\mathrm{ad}}}\widehat{\mathcal T}_n
 =\frac{p_n^+}{\overline W_n}
 \le\frac1{\overline W_n}\longrightarrow e^{\zeta(2)}.
\]
For nonnegative variables, the lower bound
\eqref{eq:tree-lower-short} and this mean upper bound force
\eqref{eq:tree-L1-short}: the negative parts relative to the limit
are bounded and vanish in probability; convergence of the means then
forces the positive parts to vanish.  Thus
\(\mathcal L_n=\overline W_n\widehat{\mathcal T}_n\to1\) in \(L^1(\mathbb P_n^{\mathrm{ad}})\),
and \(1-p_n^+\le\E_{\mathbb P_n^{\mathrm{ad}}}|\mathcal L_n-1|\to0\).

On the disjoint positive- and zero-degree sectors, direct integration
of the density difference gives
\[
 2\|\mathbb P_n^{\mathrm{can}}-\mathbb P_n^{\mathrm{ad}}\|_{\rm TV}
 =\E_{\mathbb P_n^{\mathrm{ad}}}|\mathcal L_n-1|+(1-p_n^+),
\]
which proves \eqref{eq:second-tv-short}.  The last assertion follows
from the triangle inequality and \cref{prop:first-tv-short}.
\end{proof}

\section{The two Gaussian fluctuations}
\label{sec:fluct-short}

This section separates the two sources of fluctuation.  Once the
potential field \(U\) is fixed, the row variables are independent, and
their centered sum produces the \emph{row fluctuation}.  Their
conditional means still move with \(U\); this produces the
\emph{environment fluctuation}.

The first two subsections identify the limiting row law and justify
replacing the finite row moments by their logistic limits.  The rest of
the section studies the environment term.  Its main steps are to
linearize the mean row response around the logistic law and then to
write this linear response as a sum over the independent exponential
gaps.  The intervening remainder and gradient estimates make these two
steps rigorous.  The final proposition collects the resulting
environment central limit theorem, and the two Gaussian contributions
are combined in the next section.

\subsection{The continuum row law and its variance}

For a probability measure \(\varpi\) with finite first moment, define
\[
 \Lambda_\varpi(a)=\int(a+v)_+\,\varpi(\dd v),
\]
\begin{equation}\label{eq:continuum-survival-short}
 \overline F_\varpi(d\mid u)=
 \exp\{-\Lambda_\varpi(d-u)+\Lambda_\varpi(-u)\},
 \qquad d\ge0,
\end{equation}
and
\[
 m_\varpi(u)=\int_0^\infty\overline F_\varpi(d\mid u)\,\dd d,
 \qquad
 v_\varpi(u)=2\int_0^\infty d\overline F_\varpi(d\mid u)\,\dd d
 -m_\varpi(u)^2.
\]
Put \(\mathcal M(\varpi)=\int m_\varpi(u)\,\varpi(\dd u)\) and
\(\mathcal V(\varpi)=\int v_\varpi(u)\,\varpi(\dd u)\).

\begin{lemma}[Logistic row moments]\label{lem:logistic-moments-short}
For the logistic law \(\ell(u)\dd u\), let \(x=L(u)\) and
\(b(x)=-\log(1-x)\).  Then
\begin{equation}\label{eq:logistic-row-short}
 \overline F_\ell(d\mid u)=\frac{L(u-d)}{L(u)},\qquad
 m_\ell(u)=\frac{b(x)}x,\qquad
 \mathcal M(\ell)=\zeta(2),
\end{equation}
and
\begin{equation}\label{eq:vrow-short}
 v_{\rm row}:=\mathcal V(\ell)=4\zeta(3)-2\zeta(2).
\end{equation}
\end{lemma}

\begin{proof}
For the logistic law,
\(\Lambda_\ell(a)=\log(1+e^a)\), proving the survival formula.  The
substitution \(z=L(u-d)\) gives
\[
 m_\ell(u)=\frac1x\int_0^x\frac{\dd z}{1-z}=\frac{b(x)}x.
\]
Since \(L(U)\) is uniform when \(U\) is logistic, its average is
\(\int_0^1b(x)x^{-1}\dd x=\zeta(2)\).  The same substitution in the
second-moment survival integral yields
\[
 \E(D^2\mid x)=\frac{2\operatorname{Li}_2(x)+b(x)^2}{x}.
\]
Therefore
\[
 \mathcal V(\ell)=
 2\int_0^1\frac{\operatorname{Li}_2(x)}x\dd x
 +\int_0^1\frac{b(x)^2}x\dd x
 -\int_0^1\frac{b(x)^2}{x^2}\dd x.
\]
Termwise integration gives respectively
\(2\zeta(3),2\zeta(3),2\zeta(2)\), proving
\eqref{eq:vrow-short}.
\end{proof}

We now return to the finite row law.  For
\(\varpi_n=n^{-1}\sum_j\delta_{U_j}\), put
\begin{equation}\label{eq:H-short}
 H_i(d;U)=\Lambda_{\varpi_n}(d-U_i)-\Lambda_{\varpi_n}(-U_i)
 =\frac1n\sum_j\{(d+U_j-U_i)_+-(U_j-U_i)_+\}.
\end{equation}
As quantified by the logarithmic comparison
\eqref{eq:product-log-short} below, \(H_i(d;U)\) is the first-order
approximation to
\(-\log\{S_{i,U}(d)/S_{i,U}(0)\}\), obtained by replacing
\(-\log(1-x)\) with \(x\) in the finite product.
We then set
\begin{equation}\label{eq:Phi-short}
 \Phi_n(U)=n\mathcal M(\varpi_n)
 =\sum_i\int_0^\infty e^{-H_i(d;U)}\,\dd d.
\end{equation}

\subsection{Replacing finite products by continuum survivals}

\begin{lemma}[Finite-product replacement]\label{lem:finite-row-short}
Under \(\mathbb P_n^{\rm gap}\),
\begin{equation}\label{eq:finite-mean-short}
 \frac1{\sqrt n}\left\{\sum_im_{i,n}(U)-\Phi_n(U)\right\}
 \longrightarrow0
\end{equation}
in \(L^1\), and
\begin{equation}\label{eq:finite-var-short}
 \frac1n\sum_iv_{i,n}(U)\longrightarrow v_{\rm row}
\end{equation}
in probability and in \(L^1\).
\end{lemma}

\begin{proof}
Write
\(\widetilde F_{i,n}(d)=S_{i,U}(d)/S_{i,U}(0)\). We first omit the upper-endpoint correction \(S_{i,U}(n-)\) in the
exact survival formula and define the resulting auxiliary moments 
\[
 \widetilde m_{i,n}
 =\int_0^\infty\widetilde F_{i,n}(d)\,\dd d,\qquad
 \widetilde M_{2,i,n}
 =2\int_0^\infty d\,\widetilde F_{i,n}(d)\,\dd d.
\]
Here and below the finite survival is extended by zero for \(d\ge n\).
For the continuum survival in
\eqref{eq:continuum-survival-short}, write
\[
 M_{2,\varpi}(u)
 =2\int_0^\infty d\,\overline F_\varpi(d\mid u)\,\dd d.
\]
When
\(R+d\le n/2\), put
\[
 a_j=\frac{(d+U_j-U_i)_+}{n},\qquad
 b_j=\frac{(U_j-U_i)_+}{n}.
\]
Then \(0\le a_j,b_j\le1/2\), and
\[
 |\log(1-a_j)-\log(1-b_j)+(a_j-b_j)|
 \le a_j^2+b_j^2.
\]
Summing over \(j\ne i\) costs at most
\(2(R+d)^2/n\).  The self index \(j=i\), present in \(H_i\) but
absent from the finite product, contributes \(d/n\).  Hence,
uniformly in \(i\),
\begin{equation}\label{eq:product-log-short}
 \left|\log\widetilde F_{i,n}(d)+H_i(d;U)\right|
 \le\frac{3(R+d+1)^2}{n}.
\end{equation}

We also record the integrable envelope needed to turn this logarithmic
comparison into moment bounds.  For every
\(x\in[-R,R]\),
\[
 (d+x)_+-x_+\ge(d-R)_+,
\]
and therefore
\[
 H_i(d;U)\ge(d-R)_+,\qquad
 \int_0^\infty e^{-H_i(d;U)}\,\dd d\le R+1.
\]
On the event \(R\le A\log n\), integrate this comparison up to
\(R+B\log n\).  On this range the right side of
\eqref{eq:product-log-short} is \(O(\log^2n/n)\); for large \(n\),
\(|e^\varepsilon-1|\le2|\varepsilon|\) therefore converts
\eqref{eq:product-log-short} and the preceding envelope into an
integrated error \(O(\log^3n/n)\).  Multiplication by \(d\) costs one
further factor \(O(\log n)\).  Beyond \(R\), the two survivals are bounded by
\(e^{-(d-R)/2}\) and \(e^{-(d-R)}\), respectively.  Thus
\begin{align}
 |\widetilde m_{i,n}-m_{\varpi_n}(U_i)|
 &\le C\frac{\log^3n}{n}+O(n^{-B/2}),\label{eq:mean-bound-short}\\
 |\widetilde M_{2,i,n}-M_{2,\varpi_n}(U_i)|
 &\le C\frac{\log^4n}{n}+O(n^{-B/2}\log n).
 \label{eq:second-bound-short}
\end{align}
The upper-cap subtraction is zero except in the maximum-potential
row, where its ratio is \(q_n^{\mathrm{cap}}\le(A\log n/n)^{n-1}\).  From
\eqref{eq:row-survival-short}, its pointwise effect is at most
\(q_n^{\mathrm{cap}}/(1-q_n^{\mathrm{cap}})\); after integration, its effect on the mean and second
moment is at most, respectively,
\[
 \frac{nq_n^{\mathrm{cap}}}{1-q_n^{\mathrm{cap}}},\qquad
 \frac{n^2q_n^{\mathrm{cap}}}{1-q_n^{\mathrm{cap}}}.
\]
These are superpolynomially small on the event, so the same bounds
hold for the exact conditioned moments.  Summing
\eqref{eq:mean-bound-short} proves \eqref{eq:finite-mean-short} on
the event.   If
\(\mathcal B_{n,A}=\{R\le A\log n\}\), then
\begin{align*}
 &\E\left[
 \frac{\1_{\mathcal B_{n,A}^{c}}}{\sqrt n}
 \sum_i|m_{i,n}-m_{\varpi_n}(U_i)|\right]\\
 &\qquad\le
 n^{3/2}\Pp(\mathcal B_{n,A}^{c})
 +\sqrt n\,\E[(R+1)\1_{\mathcal B_{n,A}^{c}}],\\
 &\E\left[
 \frac{\1_{\mathcal B_{n,A}^{c}}}{n}
 \sum_i|\E(D_i^2\mid U)-M_{2,\varpi_n}(U_i)|\right]\\
 &\qquad\le
 n^2\Pp(\mathcal B_{n,A}^{c})
 +C\E[(R+1)^2\1_{\mathcal B_{n,A}^{c}}].
\end{align*}
Here \(0\le D_i\le n\),
\(m_{\varpi_n}(U_i)\le R+1\), and
\(M_{2,\varpi_n}(U_i)\le C(R+1)^2\) were used.  Choosing the
polynomial margin in \eqref{eq:weighted-tail-short} larger than the
displayed powers of \(n\) makes both bounds \(o(1)\).  This removes
the exceptional event.

It remains to identify the averaged continuum variance.  With
\(\delta_n=W_1(\varpi_n,\ell)\), the hinge is one-Lipschitz, so
\[
 e^{-2\delta_n}\overline F_\ell(d\mid u)
 \le\overline F_{\varpi_n}(d\mid u)
 \le e^{2\delta_n}\overline F_\ell(d\mid u).
\]
On \(\delta_n\le1\), integration gives
\[
 |v_{\varpi_n}(u)-v_\ell(u)|
 \le C\delta_n\{1+(u_+)^2\}.
\]
The bound \eqref{eq:W1-short} gives
\(\E\delta_n^2=O(\log n/n)\).  We spell out the uniform
square-integrability used here.  With
\(\Delta_n=\sum_j(U_{(j)}-\bar u_{j,n})^2\),
\[
 \frac1n\sum_iU_i^2
 \le\frac2n\sum_j\bar u_{j,n}^2+\frac{2\Delta_n}{n}.
\]
The deterministic first term is uniformly bounded, while
\(\E\Delta_n^2=O(\log^2n)\) by
\eqref{eq:moment-short}.  Consequently
\[
 \sup_n\E\left(\frac1n\sum_iU_i^2\right)^2<\infty.
\]
Thus the empirical average of the last bound tends to zero in
\(L^1\) by Cauchy--Schwarz.  On \(\delta_n>1\), use
\(n^{-1}\sum_iv_{\varpi_n}(U_i)\le C(R+1)^2\), the fourth moment of
\(R\), and \(\Pp(\delta_n>1)=O(\log n/n)\).  More explicitly,
\[
 \E[(R+1)^2\1_{\{\delta_n>1\}}]
 \le \{\E(R+1)^4\}^{1/2}\Pp(\delta_n>1)^{1/2}=o(1).
\]
Finally
the explicit formulas in \cref{lem:logistic-moments-short} give
\[
 \lim_{u\to-\infty}v_\ell(u)=1,
 \qquad
 \lim_{u\to+\infty}v_\ell(u)=\frac{\pi^2}{3}.
\]
Thus \(v_\ell\) is bounded and continuous, so the logistic empirical
limit gives \(n^{-1}\sum_iv_\ell(U_i)\to\mathcal V(\ell)\).  Combine
this with \eqref{eq:second-bound-short} and the elementary comparison
\[
 |v_{i,n}-v_{\varpi_n}(U_i)|
 \le |\E(D_i^2\mid U)-M_{2,\varpi_n}(U_i)|
 +|m_{i,n}-m_{\varpi_n}(U_i)|
  \{m_{i,n}+m_{\varpi_n}(U_i)\}
\]
to prove
\eqref{eq:finite-var-short}.
\end{proof}

\subsection{First variation of the mean row response}

The fluctuation of \(\Phi_n\) is determined by the first variation
of \(\mathcal M\) at the logistic law.  We first take a compactly
supported signed measure \(h\) of total mass zero and insert
\(\ell+t h\).  The coefficient of \(t\) is the \emph{formal
first variation}.  Whenever \(\ell+t h\) is a probability for a
one-sided interval of \(t\)'s, the same coefficient is the ordinary
directional derivative.  The resulting linear functional extends to
every mass-zero signed measure with finite first absolute moment.  For
such an \(h\), put \(R_h(a):=\int(a+v)_+h(\dd v)\).

\begin{lemma}[Influence function]\label{lem:influence-short}
The formal first variation of \(\mathcal M\) at \(\ell\), and hence
every admissible probability-direction derivative, has the form
\[
 D\mathcal M_\ell[h]=\int f(v)h(\dd v),
\]
where \(f\) is unique up to an additive constant.  If \(x=L(u)\),
\begin{equation}\label{eq:reff-short}
 f'(u)=r(x):=1-\frac{(1-x)b(x)}x-\log x\log(1-x),
\end{equation}
\begin{equation}\label{eq:Aeff-short}
 \int_0^xr(y)\,\dd y
 =A(x):=\log x\{x+(1-x)\log(1-x)\},
 \qquad \int_0^1r(x)\dd x=0.
\end{equation}
\end{lemma}

\begin{proof}
Differentiating \eqref{eq:continuum-survival-short} under the
integral gives
\[
 \dot m_h(u)=m_\ell(u)R_h(-u)
 -\frac1{L(u)}\int_0^\infty L(u-d)R_h(d-u)\,\dd d.
\]
Substitute this into
\[
 D\mathcal M_\ell[h]
 =\int_{\mathbb R}m_\ell(v)h(\dd v)
  +\int_{\mathbb R}\dot m_h(u)\ell(u)\,\dd u,
\]
apply Fubini first for compactly supported \(h\), and differentiate
the coefficient of \(h(\dd v)\).  The two indicator derivatives
cancel for \(u<v\); for \(u>v\) one obtains
\[
 f'(v)=m_\ell'(v)-\int_{u>v}\ell(u)
 \int_{d>u-v}\frac{L(u-d)}{L(u)}\,\dd d\,\dd u.
\]
If \(x=L(v)\), the inner integral is \(b(x)/L(u)\), while
\(m_\ell'(v)=1-(1-x)b(x)/x\).  Changing variables \(y=L(u)\) proves
\eqref{eq:reff-short}.  Direct integration gives
\eqref{eq:Aeff-short}; \(A(0)=A(1)=0\).  The resulting \(f'\) is
bounded, so truncation extends this linear functional continuously to
every mass-zero signed measure with finite first absolute moment.
A detailed calculation is given in
\cref{app:influence-derivative}.
\end{proof}

Define the spacing coefficient profile
\begin{equation}\label{eq:Gamma-short}
 \Gamma(x)=-\frac{A(x)}{x(1-x)}.
\end{equation}
The profile \(\Gamma(x)\) is the limiting coefficient with which a
centered exponential gap of rank asymptotic to \(nx\) enters the
linearized fluctuation of \(n^{-1/2}\Phi_n\); its squared
\(L^2(0,1)\)-norm is the environment variance.

\subsection{The exponential-gap triangular array}

By \eqref{eq:ordered-short}, the displacement of the potential field
from its harmonic grid is a linear function of the independent centered
gaps \(E_k-1\), and the coefficients below are exactly those obtained
by applying the first variation to that displacement.  Thus the next
lemma proves the central limit theorem for the linearized environment
response; the following remainder and gradient estimates show that it
also governs the original functional \(\Phi_n\).

\begin{lemma}[Environment triangular array]\label{lem:env-array-short}
Put
\[
 a_{n,k}=\frac1{\sqrt n}\left[
 \frac1{n-k}\sum_{j>k}f'(\bar u_{j,n})
 -\frac1k\sum_{j\le k}f'(\bar u_{j,n})\right].
\]
Then
\begin{equation}\label{eq:array-short}
 \max_k|a_{n,k}|\to0,\qquad
 \sum_ka_{n,k}^2\to
 v_{\rm env}:=\int_0^1\Gamma(x)^2\dd x
 =6\zeta(2)-8\zeta(3),
\end{equation}
and
\begin{equation}\label{eq:array-clt-short}
 \sum_{k=1}^{n-1}a_{n,k}(E_k-1)
 \Longrightarrow\mathcal N(0,v_{\rm env}).
\end{equation}
\end{lemma}

\begin{proof}
The harmonic grid satisfies
\(L(\bar u_{j,n})\asymp j/n\) and
\(1-L(\bar u_{j,n})\asymp(n+1-j)/n\), uniformly in \(j\), and has
the sharper approximation \(L(\bar u_{j,n})=j/n+O_\varepsilon(n^{-1})\)
on central indices.  Riemann sums and \eqref{eq:Aeff-short} therefore
give, uniformly for \(\varepsilon n\le k\le(1-\varepsilon)n\),
\[
 \sqrt n\,a_{n,k}\longrightarrow
 \frac1{1-x}\int_x^1r(y)\dd y-
 \frac1x\int_0^xr(y)\dd y
 =\Gamma(x),\qquad x=k/n.
\]
The endpoint behavior is explicit:
\[
 r(x)=x\log x+O(x)\quad(x\downarrow0),
 \qquad r(x)\longrightarrow1\quad(x\uparrow1),
\]
and, from \eqref{eq:Aeff-short}--\eqref{eq:Gamma-short},
\[
 \Gamma(x)\longrightarrow0\quad(x\downarrow0),
 \qquad \Gamma(x)\longrightarrow1\quad(x\uparrow1).
\]
Thus both discrete averages and \(\Gamma\) are uniformly bounded on
the endpoint blocks.
Letting \(\varepsilon\downarrow0\) proves the first two assertions in
\eqref{eq:array-short}.  Lindeberg--Feller applies to
\(a_{n,k}(E_k-1)\): if \(\alpha_n=\max_k|a_{n,k}|\), its Lindeberg
sum is bounded by
\[
 \left(\sum_ka_{n,k}^2\right)
 \E\bigl[(E_1-1)^2\1_{\{|E_1-1|>\varepsilon/\alpha_n\}}\bigr]
 \longrightarrow0.
\]

For the constant, expand positively
\[
 \frac{x+(1-x)\log(1-x)}{x(1-x)}
 =\sum_{k\ge1}\frac{k}{k+1}x^k
\]
and use \(\int_0^1x^m\log^2x\dd x=2/(m+1)^3\).  Grouping by
\(m=k+l\) gives
\[
 v_{\rm env}=2\sum_{m\ge2}\frac1{(m+1)^3}
 \left[m-1-\frac{2(m+1)}{m+2}(\mathsf H_m-1)\right].
\]
The two elementary sums
\[
 \sum_{m\ge2}\frac{m-1}{(m+1)^3}
 =\zeta(2)-2\zeta(3)+1,
\]
\[
 \sum_{m\ge2}\frac{\mathsf H_m-1}{(m+1)^2(m+2)}
 =\zeta(3)-\zeta(2)+\frac{1}{2}
\]
then yield \(6\zeta(2)-8\zeta(3)\).  The second identity follows by
partial fractions, summation by parts, and Euler's
\(\sum_{r\ge1}\mathsf H_r/r^2=2\zeta(3)\).
\end{proof}

To connect the continuum derivative to the finite functional, we use
the following direct hinge estimate.  It is the only nonlinear
remainder estimate in the proof.

\subsection{A direct nonlinear remainder bound}

\begin{lemma}[Direct two-cloud linearization]\label{lem:linearization-short}
If \(U,V\in\R^n\) have the same strict coordinate order, \(z=U-V\),
and \(\bar z=n^{-1}\sum_i z_i\), then
\begin{equation}\label{eq:hinge-linear-short}
 |\Phi_n(U)-\Phi_n(V)-D\Phi_n(V)[z]|
 \le C\{1+\osc U+\osc V\}
 \|z-\bar z\mathbf1\|_2^2.
\end{equation}
\end{lemma}

\begin{proof}
Relabel so that \(V_1<\cdots<V_n\) and \(U_1<\cdots<U_n\).  For
\(k<i\), put \(t_{ik}=V_i-V_k\), \(h_{ik}=z_i-z_k\).  Then
\[
 H_i(d;V)=\frac{n-i+1}{n}d+
 \frac1n\sum_{k<i}(d-t_{ik})_+.
\]
For one moving threshold,
\[
 R_{\rm hinge}(d;t,h)
 =(d-t-h)_+-(d-t)_++h\1_{\{d>t\}}
\]
is nonnegative and is supported on the interval with endpoints
\(t\) and \(t+h\).  Indeed, with \(y=d-t\),
\[
 R_{\rm hinge}=\begin{cases}
 h-y,& h>0,\ 0<y<h,\\
 y-h,& h<0,\ h<y<0,\\
 0,&\text{otherwise}.
 \end{cases}
\]
The common-order assumption gives \(t>0\) and \(t+h>0\), and the
two triangles have area \(h^2/2\).  Therefore
\[
 H_i(d;U)-H_i(d;V)=L_i(d)+R_i(d),
\]
where
\[
 L_i(d)=-\frac1n\sum_{k<i}h_{ik}\1_{\{d>t_{ik}\}},\qquad
 R_i(d)=\frac1n\sum_{k<i}R_{\rm hinge}(d;t_{ik},h_{ik}),
\]
and \(\int R_i=(2n)^{-1}\sum_{k<i}h_{ik}^2\).  Difference quotients
in the ordered-hinge formula converge away from the finitely many
thresholds and are dominated by an integrable exponential tail.
Thus
\[
 D\Phi_n(V)[z]
 =-\sum_i\int_0^\infty e^{-H_i(d;V)}L_i(d)\,\dd d.
\]
Indeed, along a sufficiently short segment \(V+sz\),
\(H_i(d;V+sz)\ge(d-R_0)_+\) for a fixed \(R_0\), while the hazard
difference quotient is bounded by
\(c_i:=n^{-1}\sum_{k<i}|h_{ik}|\).  The mean-value theorem therefore
gives the explicit integrable domination
\[
 \left|
 \frac{e^{-H_i(d;V+sz)}-e^{-H_i(d;V)}}s
 \right|
 \le c_i e^{-(d-R_0)_+},
 \qquad |s|\le s_0.
\]
This proves the differentiation under the integral by dominated
convergence.

The elementary
inequality, for \(a,b\ge0\),
\[
 0\le e^{-b}-e^{-a}+e^{-a}(b-a)
 \le\frac{(b-a)^2}{2}(e^{-a}+e^{-b})
\]
applied with \(a=H_i(d;V)\), \(b=H_i(d;U)\) gives
\[
 e^{-b}-e^{-a}+e^{-a}L_i
 =\{e^{-b}-e^{-a}+e^{-a}(b-a)\}-e^{-a}R_i.
\]
Both the term in braces and \(R_i\) are nonnegative.  Hence the
absolute value is bounded by the sum of the quadratic exponential
remainder and \(e^{-a}R_i\).  Together with the bound
\[
 |H_i(d;U)-H_i(d;V)|^2
 \le\frac1n\sum_{k<i}h_{ik}^2
\]
this shows, after integrating and using
\(\int e^{-H_i(d;W)}\dd d\le1+\osc W\), that the Taylor remainder is
at most
\[
 \frac{C(1+\osc U+\osc V)}n\sum_{i>k}(z_i-z_k)^2.
\]
The identity
\(\sum_{i>k}(z_i-z_k)^2=n\sum_i(z_i-\bar z)^2\) proves
\eqref{eq:hinge-linear-short}.
\end{proof}

\subsection{The finite gradient on the harmonic grid}

\begin{lemma}[Harmonic-grid gradient]\label{lem:grid-gradient-short}
At the harmonic grid,
\begin{equation}\label{eq:gradient-short}
 \frac1n\sum_{j=1}^n
 \left|\partial_{U_j}\Phi_n(\bar u_n)
 -r\{L(\bar u_{j,n})\}\right|^2\longrightarrow0.
\end{equation}
\end{lemma}

\begin{proof}
Let
\[
 G_{i,U}(s):=\int_s^\infty e^{-H_i(d;U)}\,\dd d
\]
denote the integrated continuum survival in row \(i\) beyond level
\(s\).  Direct
differentiation gives
\begin{equation}\label{eq:exact-gradient-short}
 \partial_{U_k}\Phi_n(U)=
 \frac1n\sum_{j:U_j<U_k}G_{k,U}(U_k-U_j)
 -\frac1n\sum_{i:U_i>U_k}G_{i,U}(U_i-U_k).
\end{equation}
By \eqref{eq:W1-short}, the empirical harmonic grid is
\(O(\log n/n)\) from the logistic law in \(W_1\).  Hinge Lipschitz
continuity therefore compares each
survival in \eqref{eq:exact-gradient-short}, relatively and uniformly,
with \(L(u-d)/L(u)\).  The self index \(j=i\) is included here;
isolating it changes the hinge by \(d/n\), which is
\(O(\log n/n)\) on the logarithmic range, while the exponential
envelope controls the tail.  The same pointwise multiplicative bounds
pass to \(G_{i,U}\) by integration.  Since, for \(v<u\),
\[
 \int_{u-v}^\infty\frac{L(u-d)}{L(u)}\dd d
 =\frac{b\{L(v)\}}{L(u)},
\]
write \(x_j=L(\bar u_{j,n})\).  The two sums in
\eqref{eq:exact-gradient-short} are Riemann sums for
\[
 \frac1x\int_0^xb(y)\dd y-b(x)\int_x^1\frac{\dd y}{y}=r(x).
\]
We spell out the endpoint control, because convergence on compact
subintervals would not by itself imply the stated \(L^2\) result.
Here \(\mathsf H_0=0\); for \(m\ge1\), the elementary bounds
\[
 \log(m+1)\le \mathsf H_m\le1+\log m
\]
apply.  Treating the \(m=0\) endpoint separately, they imply, for
every \(j\),
\[
 e^{-2}\frac{j}{n+1}\le x_j\le
 e^2\frac{j}{n+1},\qquad
 e^{-2}\frac{n+1-j}{n+1}\le1-x_j\le
 e^2\frac{n+1-j}{n+1}.
\]
Using \(b(x)\le2x\) for \(x\le1/2\), and the second displayed
bound otherwise, split at \(n/2\) to obtain, for example,
\[
 \frac1{nx_k}\sum_{j<k}b(x_j)\le20e^4,
 \qquad
 \frac{b(x_k)}n\sum_{i>k}\frac1{x_i}\le20e^4.
\]
Indeed, the first bound reduces to the bounded averages of
\(1+\log\{(n+1)/(n+1-j)\}\).  For the second, if \(k\le n/2\), use
\((k/n)\sum_{i>k}i^{-1}\le1\); if \(k>n/2\), use
\((n-k)n^{-1}\{1+\log(n/(n-k))\}\le1\).  The numerical constant is
deliberately loose but independent of \(n\) and \(k\).

Without using \cref{lem:influence-short}, the finite-gradient
calculation gives on central indices
\[
 \widetilde r(x)=\frac1x\int_0^xb(y)\,\dd y
 -b(x)\int_x^1\frac{\dd y}{y}.
\]
Since
\(\int_0^xb(y)\dd y=x+(1-x)\log(1-x)\) and
\(\int_x^1y^{-1}\dd y=-\log x\), direct simplification gives
\(\widetilde r(x)=r(x)\) from \eqref{eq:reff-short}.
On \(\varepsilon n\le k\le(1-\varepsilon)n\), the
sharper estimate \(x_j=j/n+O_\varepsilon(n^{-1})\), after splitting
off the two endpoint pieces, gives uniform Riemann-sum convergence
with error \(o(1)\) for each fixed \(\varepsilon\).  Since \(r\) is
bounded, the displayed
uniform bounds give
\[
 \limsup_{n\to\infty}\frac1n
 \sum_{k\notin[\varepsilon n,(1-\varepsilon)n]}
 \left|\partial_{U_k}\Phi_n(\bar u_n)-r(x_k)\right|^2
 \le2\varepsilon\bigl(2\cdot20e^4+\|r\|_\infty\bigr)^2.
\]
First let \(n\to\infty\), then
\(\varepsilon\downarrow0\), to prove \eqref{eq:gradient-short}.
\end{proof}

\begin{remark}[Independent finite-gradient check]
\label{rem:finite-gradient-check}
The identity \(\widetilde r=r\) above is an independent check of the
continuum influence calculation in \cref{lem:influence-short}.
Indeed, changing one potential \(U_k\) has two effects in the finite
functional: it changes the row whose source is \(k\), and it changes
the appearance of \(U_k\) as a target in every other row.  These are
the two sums, with opposite signs, in
\eqref{eq:exact-gradient-short}.  Passing that exact finite gradient
to its Riemann-sum limit produces \(\widetilde r\), whereas
\cref{lem:influence-short} differentiates the continuum functional
directly and produces \(r\).  Their algebraic equality therefore
checks both occurrences of the empirical potential law; it is not
being deduced from the continuum differentiation.
\end{remark}

\begin{lemma}[Two-sided discrete Hardy bound]\label{lem:hardy-short}
For every \(e_1,\ldots,e_n\in\R\),
\[
 \sum_{k=1}^{n-1}\left|\frac1{\sqrt n}\left\{
 \frac1{n-k}\sum_{j>k}e_j-
 \frac1k\sum_{j\le k}e_j\right\}\right|^2
 \le16\frac1n\sum_{j=1}^n|e_j|^2.
\]
\end{lemma}

\begin{proof}
The one-sided discrete Hardy inequality
\cite[Theorem~326]{HardyLittlewoodPolya1952} says
\[
 \sum_{k=1}^n\left|\frac1k\sum_{j\le k}e_j\right|^2
 \le4\sum_{j=1}^n|e_j|^2.
\]
Apply it once in the forward order and once to the reversed sequence.
Then use \(|a-b|^2\le2|a|^2+2|b|^2\) and multiply by \(1/n\).
The two contributions are at most \(8n^{-1}\sum_j|e_j|^2\) each,
which proves the displayed constant \(16\).
\end{proof}

\subsection{The environment fluctuation and its centering}

\begin{proposition}[Environment response]\label{prop:environment-short}
Under \(\mathbb P_n^{\rm gap}\),
\begin{equation}\label{eq:Phi-clt-short}
 \frac{\Phi_n(U)-\E\Phi_n(U)}{\sqrt n}
 \Longrightarrow\mathcal N(0,v_{\rm env}),
 \qquad
 \E\Phi_n(U)=n\zeta(2)+O(\log^2n).
\end{equation}
\end{proposition}

\begin{proof}
The segment joining \(U\) to \(\bar u_n\) remains in one ordering
region.  By \cref{lem:gaps-short,lem:linearization-short},
\begin{equation}\label{eq:remainder-L1-short}
 \E|\Phi_n(U)-\Phi_n(\bar u_n)
 -\nabla\Phi_n(\bar u_n)\cdot(U-\bar u_n)|=O(\log^2n).
\end{equation}
From \eqref{eq:ordered-short}, the linear term divided by \(\sqrt n\)
is the spacing array in \cref{lem:env-array-short}, with
\(f'(\bar u_{j,n})\) replaced by the finite gradient.
\Cref{lem:hardy-short} and \eqref{eq:gradient-short} show that this replacement is
negligible in coefficient \(\ell^2\).  Hence
\eqref{eq:array-clt-short} and \eqref{eq:remainder-L1-short} prove the
first assertion in \eqref{eq:Phi-clt-short}.

Finally, let
\[
 \bar\varpi_n=\frac1n\sum_j\delta_{\bar u_{j,n}},\qquad
 \delta_n^{\rm grid}=W_1(\bar\varpi_n,\ell)
 =O(\log n/n).
\]
Since \(m_\ell(u)\le1+u_+\) and the positive part of the harmonic
grid has a bounded empirical mean,
\[
 \frac1n\sum_jm_\ell(\bar u_{j,n})=O(1).
\]
The relative-survival comparison used in
\cref{lem:finite-row-short} gives
\[
 \left|\Phi_n(\bar u_n)-\sum_jm_\ell(\bar u_{j,n})\right|
 \le Cn\delta_n^{\rm grid}=O(\log n).
\]
For completeness, set
\[
 \lambda_{j,n}=\log\frac{j}{n+1-j},
 \qquad L(\lambda_{j,n})=\frac{j}{n+1}.
\]
The harmonic-number expansion gives
\[
 |\bar u_{j,n}-\lambda_{j,n}|
 \le C\left(\frac1j+\frac1{n+1-j}\right),
 \qquad
 \sum_j|\bar u_{j,n}-\lambda_{j,n}|=O(\log n).
\]
Moreover \(m_\ell'(u)=1-(1-L(u))b\{L(u)\}/L(u)\) lies in
\([0,1]\), so replacing the harmonic grid by \(\lambda_{j,n}\)
costs only \(O(\log n)\).  Now put \(h(x)=b(x)/x\).  This function is
increasing.  The upper-minus-lower Riemann-sum error on
\([0,n/(n+1)]\), after multiplication by \(n+1\), is at most
\[
 h\!\left(\frac n{n+1}\right)-h(0)=O(\log n).
\]
The omitted last interval also contributes \(O(\log n)\), because
\((n+1)\int_{n/(n+1)}^1h(x)\,\dd x=O(\log n)\).  Hence
\[
 \sum_jm_\ell(\bar u_{j,n})
 =n\int_0^1\frac{b(x)}x\dd x+O(\log n)
 =n\zeta(2)+O(\log n).
\]
The linear term has mean zero, and
\eqref{eq:remainder-L1-short} gives the centering assertion.
\end{proof}

\section{Completion of the proof}
\label{sec:ProofComp}

\begin{proposition}[Reference-law CLT]\label{prop:B-clt-short}
Under \(\mathbb P_n^{\mathrm{ref}}\),
\[
 \sqrt n\left\{\frac1n\sum_iD_i-\zeta(2)\right\}
 \Longrightarrow
 \mathcal N\bigl(0,4\zeta(2)-4\zeta(3)\bigr).
\]
\end{proposition}

\begin{proof}
Insert \(\Phi_n(U)\) into the exact decomposition
\begin{align*}
 \sqrt n\left\{\frac1n\sum_iD_i-\zeta(2)\right\}
={}&\frac1{\sqrt n}\sum_i(D_i-m_{i,n})
 +\frac1{\sqrt n}\left\{\sum_im_{i,n}-\Phi_n(U)\right\}\\
 &+\frac{\Phi_n(U)-\E\Phi_n(U)}{\sqrt n}
 +\frac{\E\Phi_n(U)-n\zeta(2)}{\sqrt n}.
\end{align*}
The second term vanishes by \cref{lem:finite-row-short}, the fourth
by \cref{prop:environment-short}, and the third
converges to \(\mathcal N(0,v_{\rm env})\).  By
\eqref{eq:conditional-cf-short} and \eqref{eq:finite-var-short}, the
conditional characteristic function of the first term converges in
probability, hence in \(L^1\), to
\(e^{-t^2v_{\rm row}/2}\).  If \(X_n\) is the first term and \(Y_n\)
the remaining \(U\)-measurable terms, then
\[
 \left|\E e^{it(X_n+Y_n)}-
 e^{-t^2v_{\rm row}/2}\E e^{itY_n}\right|
 \le\E\left|\E(e^{itX_n}\mid U)-e^{-t^2v_{\rm row}/2}\right|
 \longrightarrow0.
\]
Thus the variances add, without an independence assertion:
\[
 v_{\rm row}+v_{\rm env}
 =4\zeta(3)-2\zeta(2)+6\zeta(2)-8\zeta(3)
 =4\zeta(2)-4\zeta(3).
\]
\end{proof}

\begin{proof}[Proof of \cref{thm:main}]
For every fixed \(t\), total variation and
\cref{prop:first-tv-short,prop:second-tv-short} give
\[
 \left|\E_{\mathbb P_n^{\mathrm{can}}}e^{it\sqrt n(n^{-1}\sum_iD_i-\zeta(2))}
 -\E_{\mathbb P_n^{\mathrm{ref}}}e^{it\sqrt n(n^{-1}\sum_iD_i-\zeta(2))}\right|
 \le2\|\mathbb P_n^{\mathrm{can}}-\mathbb P_n^{\mathrm{ref}}\|_{\rm TV}\longrightarrow0.
\]
Only a modulus-one test is transferred; no unbounded moment is passed
through total variation.  Under the exact canonical law,
\(C_n=n^{-1}\sum_iD_i\) by \cref{prop:canonical}.  The conclusion
follows from \cref{prop:B-clt-short} and L\'evy's theorem.
\end{proof}

\appendix

\section{Derivative calculation for the influence function}
\label{app:influence-derivative}

We give the details behind \cref{lem:influence-short}.  Let \(h\) be a
compactly supported signed measure with \(h(\mathbb R)=0\), and write
\[
 \varpi_t(\dd u)=\ell(u)\,\dd u+t h(\dd u).
\]
The calculation is algebraic when \(\varpi_t\) is signed and is the
usual one-sided derivative whenever \(\varpi_t\) remains a probability
measure.  Recall
\[
 R_h(a)=\int_{\mathbb R}(a+v)_+h(\dd v).
\]
Since \(\Lambda_\varpi\) is linear in \(\varpi\),
\[
 \left.\frac{\dd}{\dd t}\Lambda_{\varpi_t}(a)\right|_{t=0}
 =R_h(a).
\]
Differentiating the exponent in
\eqref{eq:continuum-survival-short} therefore gives
\begin{align*}
 \left.\frac{\dd}{\dd t}
 \overline F_{\varpi_t}(d\mid u)\right|_{t=0}
 &={\overline F}_\ell(d\mid u)
   \{R_h(-u)-R_h(d-u)\}.
\end{align*}
Using
\({\overline F}_\ell(d\mid u)=L(u-d)/L(u)\) and integrating over
\(d\) yields
\begin{align}
 \dot m_h(u)
 &:=\left.\frac{\dd}{\dd t}m_{\varpi_t}(u)\right|_{t=0}
 \notag\\
 &=m_\ell(u)R_h(-u)
 -\frac1{L(u)}\int_0^\infty
 L(u-d)R_h(d-u)\,\dd d.
 \label{eq:appendix-mdot}
\end{align}

There are two contributions when \(\mathcal M(\varpi_t)\) is
differentiated: the row response changes, and the measure with respect
to which it is averaged changes.  Consequently,
\begin{equation}
 D\mathcal M_\ell[h]
 =\int_{\mathbb R}m_\ell(v)h(\dd v)
 +\int_{\mathbb R}\dot m_h(u)\ell(u)\,\dd u.
 \label{eq:appendix-Mdot}
\end{equation}
Now
\[
 R_h(-u)=\int_{\mathbb R}(v-u)_+h(\dd v),
 \qquad
 R_h(d-u)=\int_{\mathbb R}(d-u+v)_+h(\dd v).
\]
Substituting these identities into
\eqref{eq:appendix-mdot}--\eqref{eq:appendix-Mdot} and applying Fubini
shows that
\[
 D\mathcal M_\ell[h]=\int_{\mathbb R}f(v)h(\dd v),
\]
where one may take
\begin{align}
 f(v)={}&m_\ell(v)
 +\int_{\mathbb R}\ell(u)m_\ell(u)(v-u)_+\,\dd u
 \notag\\
 &-\int_{\mathbb R}\frac{\ell(u)}{L(u)}
   \int_0^\infty L(u-d)(d-u+v)_+\,\dd d\,\dd u.
 \label{eq:appendix-f}
\end{align}
All the integrals are absolutely convergent for compactly supported
\(h\).  Because \(h(\mathbb R)=0\), adding a constant to \(f\) does
not change the linear functional.

We next differentiate \eqref{eq:appendix-f} with respect to \(v\).
Away from a null set of thresholds,
\[
 \frac{\partial}{\partial v}(v-u)_+=\1_{\{u<v\}},
 \qquad
 \frac{\partial}{\partial v}(d-u+v)_+=\1_{\{d>u-v\}}.
\]
It follows that
\begin{align*}
 f'(v)={}&m_\ell'(v)
 +\int_{u<v}\ell(u)m_\ell(u)\,\dd u\\
 &-\int_{\mathbb R}\frac{\ell(u)}{L(u)}
   \int_0^\infty L(u-d)\1_{\{d>u-v\}}\,\dd d\,\dd u.
\end{align*}
If \(u<v\), then \(u-v<0\), so the inner integral runs over every
\(d\ge0\) and
\[
 \frac1{L(u)}\int_0^\infty L(u-d)\,\dd d=m_\ell(u).
\]
This cancels the second term above.  Hence
\begin{equation}
 f'(v)=m_\ell'(v)
 -\int_v^\infty\frac{\ell(u)}{L(u)}
   \int_{u-v}^\infty L(u-d)\,\dd d\,\dd u.
 \label{eq:appendix-fprime}
\end{equation}

Put \(x=L(v)\) and \(b(x)=-\log(1-x)\).  With the substitution
\(s=u-d\),
\[
 \int_{u-v}^\infty L(u-d)\,\dd d
 =\int_{-\infty}^vL(s)\,\dd s
 =\log(1+e^v)=b(x).
\]
Moreover, since \(\ell(u)=L(u)\{1-L(u)\}\),
\[
 \int_v^\infty\frac{\ell(u)}{L(u)}\,\dd u
 =\int_v^\infty\{1-L(u)\}\,\dd u
 =-\log L(v)=-\log x.
\]
Finally, \(m_\ell(v)=b(x)/x\) and
\(\dd x/\dd v=x(1-x)\), so
\[
 m_\ell'(v)
 =\left\{\frac1{x(1-x)}-\frac{b(x)}{x^2}\right\}x(1-x)
 =1-\frac{(1-x)b(x)}x.
\]
Substitution in \eqref{eq:appendix-fprime} gives
\[
 f'(v)
 =1-\frac{(1-x)b(x)}x+b(x)\log x
 =1-\frac{(1-x)b(x)}x-\log x\log(1-x),
\]
which is \eqref{eq:reff-short}.

For completeness, let
\[
 B(x)=x+(1-x)\log(1-x),
 \qquad A(x)=\log x\,B(x).
\]
Since \(B'(x)=b(x)\),
\[
 A'(x)=\frac{B(x)}x+\log x\,b(x)
 =1-\frac{(1-x)b(x)}x-\log x\log(1-x)=r(x).
\]
Also \(B(x)=O(x^2)\) as \(x\downarrow0\), while
\(\log x\to0\) as \(x\uparrow1\).  Thus \(A(0)=A(1)=0\), and
\[
 \int_0^xr(y)\,\dd y=A(x),
 \qquad
 \int_0^1r(y)\,\dd y=0,
\]
as asserted in \eqref{eq:Aeff-short}.

\end{document}